\documentclass[aos]{imsart}

\RequirePackage{amsthm,amsmath,amsfonts,amssymb}
\RequirePackage[authoryear]{natbib}
\RequirePackage[colorlinks,citecolor=blue,urlcolor=blue]{hyperref}
\RequirePackage{graphicx}

\usepackage{comment}

\startlocaldefs
\theoremstyle{plain}

\newtheorem{theorem}{Theorem}[section]
\newtheorem{lemma}[theorem]{Lemma}
\newtheorem{corollary}[theorem]{Corollary}
\theoremstyle{definition}
\newtheorem{definition}[theorem]{Definition}

\newtheorem*{remark}{Remark}
\newtheorem{assumption}{Assumption}

\newcommand{\la}{\langle}
\newcommand{\ra}{\rangle}

\newcommand{\Ccal}{\mathcal{C}}

\newcommand{\Fcal}{\mathcal{F}}

\newcommand{\Hcal}{\mathcal{H}}

\newcommand{\Tcal}{\mathcal{T}}

\newcommand{\Scal}{\mathcal{S}}
\newcommand{\Xcal}{\mathcal{X}}

\newcommand{\Kcal}{\mathcal{K}}

\newcommand{\Rb}{\mathbb{R}}

\newcommand{\Eb}{\mathbb{E}}

\newcommand{\Nb}{\mathbb{N}}

\newcommand{\what}{\widehat}
\newcommand{\Khat}{\widehat{K}}
\newcommand{\Ktrue}{K^\circ}
\newcommand{\argmin}{\mathop{\arg\min}}

\def\ba#1\ea{\begin{align*}#1\end{align*}} 
\def\banum#1\eanum{\begin{align}#1\end{align}} 

\endlocaldefs

\begin{document}

\begin{frontmatter}
\title{A Theory of Nonparametric Covariance Function Estimation\\ for Discretely Observed Data}
\runtitle{A Theory of Nonparametric Covariance Function Estimation}

\begin{aug}
\author[A, B]{\fnms{Yoshikazu}~\snm{Terada}\ead[label=e1]{terada.yoshikazu.es@osaka-u.ac.jp}}
\and
\author[A]{\fnms{Atsutomo}~\snm{Yara}}
\address[A]{Graduate School of Engineering Science, The University of Osaka\printead[presep={,\ }]{e1}}
\address[B]{Center for Advanced Integrated Intelligence Research, RIKEN}
\end{aug}

\begin{abstract}
We study nonparametric covariance function estimation for functional data observed with noise at discrete locations 
on a $d$-dimensional domain. 
Estimating the covariance function from discretely observed data is a challenging nonparametric problem, 
particularly in multidimensional settings, 
since the covariance function is defined on a product domain and thus suffers from the curse of dimensionality. 
This motivates the use of adaptive estimators, such as deep learning estimators. 
However, existing theoretical results are largely limited to estimators with explicit analytic representations, 
and the properties of general learning-based estimators remain poorly understood. 
We establish an oracle inequality for a broad class of learning-based estimators 
that applies to both sparse and dense observation regimes in a unified manner, 
and derive convergence rates for deep learning estimators over several classes of covariance functions. 
The resulting rates suggest that structural adaptation can mitigate the curse of dimensionality, 
similarly to classical nonparametric regression. 
We further compare the convergence rates of learning-based estimators with several existing procedures. 
For a one-dimensional smoothness class, deep learning estimators are suboptimal, 
whereas local linear smoothing estimators achieve a faster rate. 
For a structured function class, 
however, deep learning estimators attain the minimax rate up to polylogarithmic factors, 
whereas local linear smoothing estimators are suboptimal.
These results reveal a distinctive adaptivity--variance trade-off in covariance function estimation.
\end{abstract}

\begin{keyword}[class=MSC]
\kwd[Primary ]{62R10}
\kwd{62G05}
\kwd[; secondary ]{62G08}
\end{keyword}

\begin{keyword}
\kwd{Functional data}
\kwd{ReLU neural networks}
\kwd{Phase transition}
\end{keyword}

\end{frontmatter}

\section{Introduction}

Recent advances in measurement technology have led to the increasing availability of data 
in which each subject is observed over a domain, either continuously or at a finite set of locations, 
as commonly encountered in longitudinal and spatio-temporal studies. 
In functional data analysis (FDA), such data are viewed as realizations of random elements in a function space, 
or equivalently as sample paths of stochastic processes \citep{RamsaySilverman05,HorvathKokoszka12,HsingEubank15,WangEtAl2016}.
Estimation of the mean and covariance functions is a fundamental problem in FDA. 
It plays a central role in many subsequent methods, 
most notably functional principal component analysis \citep{HallHosseini-Nasab06, HallEtAl2006, ZhouEtAl2025} 
and functional linear regression \citep{HallHorowitz2007, YuanCai2010, ZhouEtAl2023}, 
and is also essential for reconstructing latent trajectories from discretely observed data \citep{YaoEtAl2005}.

When functional data are observed continuously over the domain, 
covariance estimation is essentially no more difficult than in the multidimensional setting, 
as the empirical covariance function is directly available.
By contrast, when measuerments are available only at discrete locations, 
the covariance function estimation becomes a challenging nonparametric estimation problem. 
In this paper, we focus on this standard and practically important discretely observed setting.

Consider a centered, square-integrable stochastic process $X(\cdot)$ on a $d$-dimensional domain $\mathcal{T}\subset \Rb^d$
with the covariance function $K^{\circ}(s,t)=\mathbb{E}\{X(s)X(t)\}\;\;(s,t\in \Tcal)$.
We observe i.i.d.\ replicates $X_1,\ldots,X_n$, but not the trajectories/surfaces themselves.
Instead, for subject $i$ we observe noisy measurements
\ba
Y_{ij}=X_i(T_{ij})+\varepsilon_{ij},\qquad j=1,\ldots,m_i,
\ea
where the observed points $T_{ij}\in\mathcal{T}$ are random, and $\varepsilon_{ij}$ are mean-zero noises.
For simplicity, the number of measurments are the same among the subjects (i.e., $m=m_1=\dots=m_n$).
Following \citet{ZhangWang2016}, we refer to the dense regime as the case where $m \to \infty$ as $n \to \infty$, 
and to the sparse regime as the case where $m$ is uniformly bounded.

Existing estimators are typically based on smoothing splines (\citealp{CaiYuan2010}) or local linear smoothing (\citealp{YaoEtAl2005,HallEtAl2006,LiHsing2010,ZhangWang2016}), 
and their statistical properties have been investigated in considerable detail. 
Local linear smoothing, for example, has the attractive feature that the uniform convergence can often be ensured (\citealp{LiHsing2010,ZhangWang2016}). 
Moreover, \cite{CaiYuan2010} derives convergence rates for the smoothing spline estimators and proves their minimax optimality.
Most of the existing literature focuses on functions defined on one-dimensional domains.
In many modern applications, however, functions are defined on multidimensional domains, 
and the covariance function is therefore defined on the corresponding product domain, 
whose dimension is doubled, thereby exacerbating the curse of dimensionality.

Recent results in nonparametric regression suggest that adaptive deep learning estimators can, 
for suitably structured function classes, mitigate the curse of dimensionality \citep{Schmidt-Hieber2019, SchmidtHieber2020, NakadaImaizumi2020, SuzukiNitanda2021}. 
This fact motivates the use of learning-based estimators for covariance function estimation. 
For instance, \cite{YanEtAl2025} established theoretical guarantees for deep learning estimators of the mean function, 
and \cite{SarkarPanaretos2022} proposed a deep learning approach to covariance function estimation.

In this paper, we study a general nonparametric least-squares (LS) approach to estimating $K^{\circ}$. 
More precisely, we consider a learning-based estimator defined as a minimizer of the least-squares loss
\begin{equation}\label{eq:intro-ls}
\frac{1}{n\,m(m-1)}\sum_{i=1}^n\sum_{1\le j\ne k\le m}\left\{Y_{ij}Y_{ik}-K(T_{ij},T_{ik})\right\}^2.
\end{equation}
This loss function is widely used (e.g., \citealp{YaoEtAl2005,HallEtAl2006,LiHsing2010,CaiYuan2010,ZhangWang2016}).

For mean function estimation, 
\cite{YanEtAl2025} adapted classical localization analysis \citep{BartlettEtAl2005, Koltchinskii2006, BlanchardEtAl2008} 
and thereby provided a general framework for analyzing learning-based estimators. 
For covariance function estimation, by contrast, an analogous framework is not yet available. 
\cite{SarkarPanaretos2022} provided a theoretical analysis of deep learning estimators 
for discretely observed data on regular grids. 
However, their analysis requires the number of grid points in each coordinate direction to diverge for consistency. 
This condition appears overly restrictive. 
In addtition, \cite{PaulPeng2009b} derived theoretical properties of a restricted maximum likelihood estimator
under very restrictive conditions, 
but their results do not attain optimal rates when the number of measurements $m$ diverges. 
Moreover, their approach cannot accommodate both sparse and dense regimes in a unified manner.

At first sight, one might expect the classical localization analysis for nonparametric regression 
to extend directly to covariance function estimation, as in the case of mean function estimation. 
A naive application of the existing theory, however, yields loose bounds when the number of measurements $m$ diverges. 
The main difficulty arises from the dependence structure induced by the least-squares loss (\ref{eq:intro-ls}), 
which is a sum of within-subject $U$-statistics. 
Unless this within-subject $U$-statistic structure is handled explicitly, sharp bounds cannot be obtained. 
We discuss this issue in detail in Section~\ref{sec:2.2}.

In this paper, we develop a framework for sharp localization analysis of covariance function estimation. 
Inspired by the prior work on empirical risk minimization of $U$-statistics (\citealp{ClemenconEtAl2008}), 
we extend Talagrand's inequality to the least-squares loss (\ref{eq:intro-ls}) for covariance function estimation. 
Combining this extension with localization analysis, 
we derive a sharp oracle inequality that treats sparse and dense regimes in a unified manner. 
As an application, we use this oracle inequality to establish, to the best of our knowledge, 
the first convergence-rate results for deep learning estimators over several classes of covariance functions.
We observe a phase transition phenomenon similar to those reported in \citet{CaiYuan2010, CaiYuan2011} and \cite{YanEtAl2025}.

Our results suggest that, under highly anisotropic smoothness, 
deep learning estimators can mitigate the curse of dimensionality in covariance function estimation, 
as also demonstrated in other estimation problems \citep{SuzukiNitanda2021,YanEtAl2025}.
In the sparse regime, 
the convergence rate of deep learning estimators matches, up to polylogarithmic factors, 
the minimax rate for nonparametric regression over $2d$-dimensional domains, 
as established in \cite{SuzukiNitanda2021}.

Interestingly, under the dense regime in the one-dimensional setting, 
for the class of smooth functions considered by \citet{ZhangWang2016}, 
deep learning estimators are suboptimal, whereas local linear smoothing attains a faster rate. 
In contrast, for the structured function classes studied by \citet{CaiYuan2010}, 
deep learning estimators attain the same minimax rate as smoothing spline estimators up to polylogarithmic factors, 
while local linear smoothing becomes suboptimal. 
Thus, these results reflect an inherent limitation of learning-based estimators for covariance function estimation.
Such behavior appears to be uncommon in classical nonparametric estimation and may be specific to covariance function estimation.

The remainder of the paper is organized as follows.
Section~\ref{sec:2} introduces the model setting and the estimators.
Section~\ref{sec:oracle} establishes a unified oracle inequality for general learning-based estimator.
In Section~\ref{sec:deep-cov}, we consider the deep learning estimator andderives its convergence rates for several classes of covariance functions and compares the results with existing work.
Section~\ref{sec:conclusion} contains the conclusion of this paper.
All proofs are deferred to the Appendix.

\section{Preliminaries}\label{sec:2}
In this section, we describe the repeated measurement model and the least-squares estimator and fix notation. 
We then discuss the challenges in the theoretical analysis of covariance function estimation.

\subsection{Model and least squares estimation}

In this paper, we employ the repeated measurment model, which is commonly assumed in existing literatrue 
(e.g., \cite{ZhangWang2016} and \cite{YanEtAl2025}).
Let $\Tcal := [0,1]^d$.
Let $X(\cdot)$ be a square integrable and centered stochastic process on $\Tcal$ (i.e., $\forall t \in \Tcal;\;\Eb[X(t)] = 0$.), 
and 
assume that $X$ is jointly measurable.
The covariance function of $X$ will be denoted by $K^\circ$, that is, 
\[
K^\circ(t,t') := \mathbb{E}\left[X(t)X(t')\right],\quad t,t' \in \Tcal.
\]
Let $X_1,\dots,X_n$ be independent and identically distributed copies of $X$.
The processes $X_i\;(i=1,\dots,n)$ are not observed directly.
Instead, for the $i$th subject, noisy measurements are observed at $m_i$ time points $T_{ij}$ $(j=1,\dots,m_i)$.
We assume that $T_{ij}$ are i.i.d. from a distribution $P_T$ on $\Tcal$ and are independent of $X_i\;(i=1,\dots,n)$.
That is, we assume the following repeated measruement model:
\banum
Y_{ij} = X_i(T_{ij}) + \epsilon_{ij}\quad
(i=1,\dots,n;\;j=1,\dots,m_i)
\label{eq:model}
\eanum
where $\epsilon_{ij}$ are independent noises with $\Eb[\epsilon_{ij}] = 0$.
Thus, we observe $\{(T_{ij},Y_{ij}): i= 1,\dots, n, j=1,\dots, m_i\}$.
For simplicity, we assume that the number of measurements is the same across subjects, that is, $m = m_1 = \cdots = m_n$, 
and that $P_T$ is the uniform distribution on $[0,1]^d$.

To estimate the true covariance function $K^\circ$, 
we consider the following general nonparametric least square (LS) estimator:
\banum
\Khat_n:= \argmin_{K\in \Kcal_{nm}}\frac{1}{nm(m-1)}\sum_{i=1}^n \mathop{\sum}_{1\le j\neq k\le m}\left\{Y_{ij}Y_{ik} -  K(T_{ij},T_{ik})\right\}^2,
\label{eq:LS}
\eanum
where $\Kcal_{nm}$ is a given class of symmetric functions.

To evaluate the accuracy of the LS estimator $\Khat_n$, 
the following mean squared error is considered:
\[
\Eb\left[ 
\big\| \Khat_n - K^\circ \big\|_{2}^2
\right],
\]
where
$
\|K\|_2^2 := \int_{\mathcal{T}\times\mathcal{T}} K^2(s,t)\,dP_T^{\otimes 2}(s,t).
$

\subsection{Difficulity of covariance function estimation}\label{sec:2.2}

We explain why classical localization analysis does not yield tight bounds for covariance function estimation, 
in contrast to mean function estimation.

In the theory of mean function estimation \citep{YanEtAl2025}, 
bounding the following quantity for $r>0$ is essential:
\[
\sup_{f \in \Fcal}\left| 
\frac{1}{nm}\sum_{i=1}^n\sum_{j=1}^m \frac{\|f-f^\circ\|_2^2 - \{f(T_{ij})-f^\circ(T_{ij})\}^2}{\|f-f^\circ\|_2^2 +r}
\right|,
\]
where $f^\circ$ is the true mean function of $X$, and $\Fcal$ is a given class of mean functions. 
Since the $T_{ij}$ are independent, each term in the sum is independent. 
Therefore, 
the classical Talagrand inequality (e.g., Theorem 3.3.16 of \cite{GineNickl2016}) can be applied to bound this quantity, 
and standard localization analysis yields a sharp oracle inequality.

In contrast, for covariance function estimation, 
bounding the following quantity for $r>0$ is essential:
\[
\sup_{K\in \Kcal_{nm}} \left| \frac{1}{nm(m-1)}\sum_{i=1}^n \sum_{1 \le j \ne k \le m}
\frac{\|K-K^\circ\|_{2}^2 - \left\{K(T_{ij},T_{ik}) -  K^\circ(T_{ij},T_{ik})\right\}^2}
{\|K-K^\circ\|_{2}^2 + r} \right|.
\]
In the sparse regime, where the number of measurements $m$ is bounded, 
the following within-subject $U$-statistic can be viewed as a single random variable $Z_i$:
\[
Z_i := \frac{1}{m(m-1)} \sum_{1 \le j \ne k \le m}
\frac{\|K - K^\circ\|_2^2 - \left\{K(T_{ij}, T_{ik}) - K^\circ(T_{ij}, T_{ik})\right\}^2}
{\|K - K^\circ\|_2^2 + r}.
\]
Thus, as in \cite{YanEtAl2025}, standard localization analysis combined with Talagrand's inequality yields an oracle inequality. 
However, in the dense regime, where the number of measurements $m$ diverges, 
this naive approach yields a looser bound because it does not exploit the fact that the observations approach the continuous regime.

We address this gap by developing a unified nonasymptotic analysis of the learning-based estimator \eqref{eq:LS} 
that applies to both sparse and dense regimes.

\section{A theory of nonparametric covariance function estimation}\label{sec:oracle}

In this section, we establish an oracle inequality to derive the convergence rates of the mean squared error 
for a general LS esimator.
We assume the following two general assumptions.

\begin{assumption}{(Boundedness)}\label{assumption:1}
There exists a positive number $B_K>0$ such that $\|K^\circ\|_\infty<B_K$ and $\forall K \in \Kcal_{nm};\; \|K\|_\infty < B_K$.
\end{assumption}

\begin{assumption}{(Sub-Gaussian)}\label{assumption:2}
There exist positive numbers $B_1>0$ and $B_2>0$ such that
\ba
\forall \lambda \in \Rb;\; \Eb\left[ e^{\lambda X_{i}(T_{ij})} \right] \le e^{\lambda^2 B_1^2/2}
\;\text{ and }\;
\Eb\left[ e^{\lambda \epsilon_{ij}} \right] \le e^{\lambda^2 B_2^2/2}
\ea
\end{assumption}

To overcome the difficulty described in Section~\ref{sec:2.2}, 
we establish the following modified Talagrand inequality for covariance function estimation. 
The lemma builds on a key representation of $U$-statistics due to \cite{ClemenconEtAl2008}.

\begin{lemma}{(Modified Talagrand inequality)}\label{lemma:modified-Talagrand}
Let $\{X_{ij}\}_{n\times m}$ be i.i.d.\ $\Scal$-valued random variables.
Let $\mathcal{F}$ be a countable class of measurable functions
$f=(f_1,\dots,f_n): \Scal\times \Scal \to \mathbb{R}^n$ such that $\|f_i\|_\infty \le U < \infty.$
Assume that $\mathbb{E}\bigl[f_i(X_{i1},X_{i2})\bigr]=0$ for any $f\in\mathcal{F}$ and any $i=1,\dots,n$.
Set
\ba
S_n
&:= \sup_{f\in\mathcal{F}}
\left|
\sum_{i=1}^n \sum_{j\neq k} f_i(X_{ij},X_{ik})
\right|
\;\;\text{ and }\;\;
\tilde{S}_{n}
:= \sup_{f\in\mathcal{F}}
\left|
\sum_{i=1}^n \sum_{j=1}^{\lfloor m/2\rfloor}
f_i\bigl(
X_{ij},
X_{i(\lfloor m/2\rfloor+j)}
\bigr)
\right|.
\ea
Moreover, assume that 
\[
U^2 \ge \sigma^2 \ge
\frac{1}{n}\frac{1}{\lfloor m/2\rfloor}
\sum_{i=1}^n \sum_{j=1}^{\lfloor m/2\rfloor}
\mathbb{E}\!\left[
f_i^2\bigl(
X_{ij},
X_{i(\lfloor m/2\rfloor+j)}
\bigr)
\right],
\]
and define
$
v_{n}
:= 2U\,\mathbb{E}_X[\tilde{S}_{n}]
+ n\lfloor m/2\rfloor\,\sigma^2.
$

Then, for any $x\ge 0$ and any $\alpha>0$, with probability at least
$1-e^{-x}$,
\[
\frac{1}{nm(m-1)}\,S_n
\le
(1+\alpha)\,
\mathbb{E}\!\left[
\frac{\tilde{S}_{n}}{n\lfloor m/2\rfloor}
\right]
+
\sqrt{
\frac{2\sigma^2}{n\lfloor m/2\rfloor}\,x
}
+
\left(
\frac{3}{2}+\frac{1}{\alpha}
\right)
\frac{Ux}{n\lfloor m/2\rfloor}.
\]
\end{lemma}
Combined with a conditioning argument, 
this lemma enables the localization analysis by controlling the empirical process associated with a sum of $U$-statistics.

Now, we introduce several definitions for localization analysis \citep{BartlettEtAl2005, Koltchinskii2006}.
\begin{definition}[Sub-root function]
Let $r^\ast > 0$. 
A nondecreasing function $\phi:[0,\infty)\to [0,\infty)$ is called a sub-root function on $[r^\ast,\infty)$ 
if $r \mapsto \phi(r)/\sqrt{r}$ is nonincreasing for $r\ge r^\ast$.
\end{definition}

\begin{definition}[Rademacher fixed point]
Let $r_{nm}^\ast > 0$ and let $\phi_{nm}:[0,\infty)\to [0,\infty)$ be a sub-root function on $[r_{nm}^\ast,\infty)$. 
Let $\{\sigma_{ij}\}$ be i.i.d. Rademacher variables.
Define
\[
\Kcal_{nm}(r):= \big\{ K \in \Kcal_{nm}\mid \|K - K^\circ\|_{2}^2 \le r\big\}.
\]
If $\phi_{nm}(r_{nm}^\ast) \le r_{nm}^\ast$ and
{\small
\[
\phi_{nm}(r)
\ge 
\mathbb{E}\left[ \sup_{K\in \Kcal_{nm}(r)}\left|
\frac{1}{n\lfloor m/2 \rfloor}\sum_{i = 1}^n  \sum_{j=1}^{\lfloor m/2\rfloor} 
\sigma_{ij} \left\{K\big(T_{ij},T_{i(\lfloor m/2 \rfloor + j)}\big) -  K^\circ\big(T_{ij},T_{i(\lfloor m/2 \rfloor+j)}\big)\right\}
\right|\right],
\]
}
then $r_{nm}^\ast$ is called the Rademacher fixed point of $\Kcal_{nm}$.
\end{definition}

Localization analysis based on the modified Talagrand inequality 
yields the following oracle inequality.

\begin{theorem}\label{theorem:oracle}
Under Assumption~1 and Assumption~2,
for $m\ge 2$, there exists a global constat $c>0$ such that
\ba
&\Eb\left[ \|\what{K}_n-K^\circ\|_{2}^2\right]
\le
c\left\{
\inf_{K\in \Kcal_{nm}}\mathbb{E}\bigl[\|K-K^\circ\|_2^2\bigr]
+
\frac{1}{n}
+
\log^2\left(n\lfloor m/2\rfloor\right)\left( \frac{1}{n\lfloor m/2\rfloor} + r_{nm}^\ast \right)
\right\},
\ea
where $r_{nm}^\ast$ is the Rademacher fixed point of $\Kcal_{nm}$.
\end{theorem}

This theorem can be seen as an extension of Theorem~2.1 in \cite{YanEtAl2025},
which establishes a sharp oracle inequality for mean function estimation, to covariance function estimation.
It also extends the theory of \cite{CaiYuan2010} to general least-squares (LS) estimators that may not admit an explicit form.
Interestingly, our inequality matches the oracle inequality of \cite{YanEtAl2025}. 
For mean function estimation, a total of $nm$ observations $\{Y_{ij}\}$ are used, 
whereas covariance estimation involves $nm^2$ pairwise products $\{Y_{ij}Y_{ik}\}$. 
However, these terms are highly redundant, as they are constructed from only $nm$ underlying observations $\{Y_{ij}\}$.
This point is also mentioned in Section~5 of \cite{CaiYuan2010}.

In contrast, when the estimator admits an explicit representation, such as local linear smoothing estimators (e.g., \cite{ZhangWang2016}), 
a more refined analysis that exploits its structure may improve the $1/(nm)$ rate to $1/(nm^2)$. 
This improvement, however, requires imposing additional structural constraints on the estimator, 
thereby sacrificing adaptivity.
We discuss this direction in Section~\ref{sec:comparison}.

\begin{remark}
In Theorem~\ref{theorem:oracle}, 
Assumption~\ref{assumption:2} is used to control tail probabilities in the concentration arguments.
Hence, it can be relaxed to a sub-Weibull assumption.
Under this weaker condition, our results remain valid with the same polynomial rates; only the
polylogarithmic factors deteriorate (i.e., $\log^{2}(n\lfloor m/2\rfloor)$ is replaced by $\log^{c}(n\lfloor m/2\rfloor)$ for some larger $c>2$).
\end{remark}

Theorem~\ref{theorem:oracle} is not directly applicable in practice, 
since the fixed point $r_{nm}^*$ depends on the population distribution, which is typically unknown. 
To address this issue, we extend Theorem~\ref{theorem:oracle} by replacing the Rademacher fixed point $r_{nm}^\ast$ 
with a quantity depending only on the complexity of $\Kcal_{nm}$.

\begin{corollary} \label{corollary:oracle-without-fixed-point}
    Under Assumption \ref{assumption:1} and Assumption \ref{assumption:2}, for $m \ge 2$, we have
    \ba
        \Eb[\| \hat{K}_n - K^\circ \|_2^2] \le c \left\{ \inf_{K \in \Kcal_{nm}} \Eb[\| K - K^\circ \|_2^2] + \frac{1}{n} + \frac{\mathrm{VCdim}(\Kcal) \log(n \lfloor m/2 \rfloor)}{n \lfloor m/2 \rfloor} \right\},
    \ea
    where $c$ is a global constant and $\mathrm{VCdim}(\Kcal_{nm})$ denotes the the Vapnik--Chervonenkis (VC) dimension of $\Kcal_{nm}$.
\end{corollary}

\section{Covariance function estimation via deep learning}\label{sec:deep-cov}

In this section, we derive convergence rates for the learning-based estimators based on deep neural networks (DNNs).
First, we consider a class that illustrates the potential to mitigate the curse of dimensionality.
Second, we consider a class that facilitates comparison with the results of \cite{CaiYuan2010}.
Finally, we compare the performance of the deep learning estimator with existing estimators in one-dimensional settings.

\subsection{Convergence rates for anisotropic Besov space}

To illustrate the potential of deep learning estimators to mitigate the curse of dimensionality, 
we consider the anisotropic Besov space introduced by \cite{SuzukiNitanda2021}. 
This space forms a rich function class that includes a wide range of functions 
and is particularly suitable for capturing heterogeneous smoothness in real-world data. 
For example, in spatiotemporal functional data, the covariance function may be very smooth 
in the temporal direction but much less smooth in the spatial direction, and vice versa.

We define the anisotropic Besov space following the notation of \cite{SuzukiNitanda2021}.
Let $\mathbb{R}_+ := \{x \in \mathbb{R} \mid x > 0\}$.
For $\beta = (\beta_1, \dots, \beta_d) \in \mathbb{R}_+^d$, we define
\begin{equation*}
\tilde{\beta} := \left( \sum_{j=1}^d \frac{1}{\beta_j} \right)^{-1},
\qquad
\bar{\beta} := \max_{j=1,\dots,d} \beta_j.
\end{equation*}
For a function $f : \mathbb{R}^d \to \mathbb{R}$, the $r$-th difference of $f$ in the direction $h \in \mathbb{R}^d$ is recursively defined by
\begin{equation*}
\Delta_h^r(f)(x) := \Delta_h^{r-1}(f)(x+h) - \Delta_h^{r-1}(f)(x), \quad \Delta_h^0(f)(x) := f(x),
\end{equation*}
for $x \in [0, 1]^d$ with $x + h \in [0, 1]^d$; otherwise, we set $\Delta_h^r(f)(x) := 0$.
Moreover, for $f \in L^p([0,1]^d)$ with $p \in (0,\infty]$, 
the $r$-th modulus of smoothness of $f$ is defined as
\begin{equation*}
w_{r,p}(f,t)
=
\sup_{h \in \mathbb{R}^d:\, \| h_i \| \le t_i}
\| \Delta_h^r(f) \|_p,
\end{equation*}
where $t = (t_1, \dots, t_d)\in \Rb_+^d$. 
Using this modulus of smoothness, 
we define the anisotropic Besov space $B^{\beta}_{p,q}([0,1]^d)$ for $\beta = (\beta_1, \dots, \beta_d) \in \mathbb{R}_+^d$.

\begin{definition}[Anisotropic Besov space]
For $0 < p,q \le \infty$, $\beta = (\beta_1, \dots, \beta_d) \in \mathbb{R}_+^d$, and $r := \max_i \beta_i + 1$, define the seminorm $| \cdot |_{B^{\beta}_{p,q}}$ by
\begin{equation*}
| f |_{B^{\beta}_{p,q}}
:=
\begin{cases}
\left(
\sum_{k=0}^{\infty}
\left[
2^k
w_{r,p}\left(
f,
\left(2^{-k/\beta_1}, \dots, 2^{-k/\beta_d}\right)
\right)
\right]^q
\right)^{1/q},
& \text{if } q < \infty, \\[1ex]
\sup_{k \ge 0}
\, 2^k
w_{r,p}\left(
f,
\left(2^{-k/\beta_1}, \dots, 2^{-k/\beta_d}\right)
\right),
& \text{if } q = \infty.
\end{cases}
\end{equation*}
The norm of the anisotropic Besov space $B^{\beta}_{p,q}([0,1]^d)$ is defined by
\begin{equation*}
\| f \|_{B^{\beta}_{p,q}}
:=
\| f \|_{p}
+
| f |_{B^{\beta}_{p,q}},
\end{equation*}
and
\begin{equation*}
B^{\beta}_{p,q}([0,1]^d)
=
\left\{
f \in L^p([0,1]^d)
\mid
\| f \|_{B^{\beta}_{p,q}} < \infty
\right\}.
\end{equation*}
\end{definition}

Intuitively, $\beta$ characterizes the smoothness in each coordinate direction. 
If $\beta_i$ is large, functions in $B^{\beta}_{p,q}$ are smooth along the $i$-th direction; 
otherwise, they may be less regular in that direction.

We consider the following class of covariance functions as a subset of the anisotropic Besov space:
\begin{equation*}
\Kcal^\beta_{p, q} := 
\left\{ 
K \in B^{(\beta, \beta)}_{p, q}(\Tcal^2) 
\mid \| K \|_{B^{(\beta, \beta)}_{p,q}} \le 1,\;\text{and }
K \text{ is symmetric and positive definite}
\right\}.
\end{equation*}
Here, The choice of $(\beta,\beta)$ for the smoothness parameters is natural, 
reflecting the symmetry of covariance functions.

We introduce a model for covariance functions based on deep feedforward neural networks.
Let $\sigma:\mathbb{R}\to\mathbb{R}$ denote an activation function, 
and we adopt the ReLU activation function $\sigma(x)=\max\{x,0\}$.
For $b = (b_1, \dots, b_r)^\top \in \mathbb{R}^r$, 
define the biased activation function $\sigma_{v} : \mathbb{R}^r \to \mathbb{R}^r$ by
\begin{equation*}
    \sigma_{b}
    \begin{pmatrix}
        y_1 \\
        \vdots \\
        y_r
    \end{pmatrix}
    =
    \begin{pmatrix}
        \sigma(y_1 - b_1) \\
        \vdots \\
        \sigma(y_r - b_r)
    \end{pmatrix}.
\end{equation*}
The architecture of a feedforward neural network is specified 
by its depth $L$ (the number of hidden layers) and a common width $W \in \mathbb{N}$. 
A deep neural network with depth $L$ and width $W$ is defined as a function of the form
\begin{equation}
F(x)
=
W_L \sigma_{b_L}
W_{L-1} \sigma_{b_{L-1}}
\cdots
W_1 \sigma_{b_1}
W_0 x,
\qquad x \in \mathbb{R}^{2d},
\label{eq:DNN_function}
\end{equation}
where $W_0 \in \mathbb{R}^{W \times 2d}$, $W_L \in \mathbb{R}^{1 \times W}$, 
$W_l \in \mathbb{R}^{W \times W}$ for $1 \le l \le L$, 
and $b_L \in \mathbb{R}$, $b_l \in \mathbb{R}^W$ are bias parameters.
Since we consider estimating the covariance function, the input dimension is taken to be $2d$.

We now define the class of sparse networks by
\begin{equation*}
    \mathcal{F}(L,W,S)
    =
    \left\{
        F \text{ of the form \eqref{eq:DNN_function}} \mid
        \sum_{j=0}^L
        \bigl(
            \lVert W_j \rVert_0
            +
            \lvert b_j \rvert_0
        \bigr)
        \leq S
    \right\},
\end{equation*}
where $b_0$ denotes the zero vector.
Here, $\| W_j \|_0$ and $\lvert b_j \rvert_0$ represent the numbers of nonzero entries of $W_j$ and $b_j$, respectively.
To ensure symmetry of the covariance function estimator, we consider the following class of covariance functions:
\begin{equation}
    \Kcal(L, W, S) := \left\{ K(\cdot, \star) = \frac{1}{2}\{h(\cdot, \star) + h(\star, \cdot)\} \mid h \in \Fcal(L, W, S) \right\}. \label{eq:DNNs}
\end{equation}
This DNN model \eqref{eq:DNNs} is much simpler than the CovNet model \citep{SarkarPanaretos2022}, 
leading to a simpler convergence analysis. 
It may seem that the estimator should be positive definite, as in CovNet. 
However, post-estimation projection always ensures that
\[
\big\| \widetilde{K}_n - K^\circ \big\|_{2} \le \big\| \widehat{K}_n - K^\circ \big\|_{2},
\]
where $\widetilde{K}_n$ is the projection of $\widehat{K}_n$ onto the class of positive definite functions.
Thus, positive definiteness need not be imposed at the estimation stage, as is also the case in existing approaches.

The following theorem establishes convergence rates of the deep learning estimator over the class $\Kcal^\beta_{p, q}$.
\begin{theorem} \label{theorem:convergence rates}
    Suppose $0 < p, q \le \infty$ and $\tilde{\beta}/2 > (1 / p - 1/2)_+$ for $\beta \in \Rb^d_+$.
    Let $\delta = 2\left(1/p - 1/2\right)_{+}$ and $\nu = (\tilde{\beta} / 2 - \delta) / 2\delta$.
    Consider the deep neural network model $\Kcal(L, W, S)$ satisfies
    \ba
        &\text{\rm (i) } L \asymp \log(n\lfloor m/2 \rfloor),\;\;
        \text{\rm (ii) } W \asymp (n\lfloor m/2 \rfloor)^{1 / (\tilde{\beta}+1)}, \text{ and}\\
        &\text{\rm (iii) } S \asymp (n\lfloor m/2 \rfloor)^{1 / (\tilde{\beta} + 1)}\log^2(n\lfloor m/2 \rfloor).
    \ea
    If $K^\circ \in \Kcal_{p, q}^\beta$ and $\| K^\circ \|_\infty \le B_K$, under Assumption \ref{assumption:2}, we obtain
    \begin{equation*}
        \Eb[\lVert \what{K}_n - K^\circ \rVert_2^2] 
        \le 
        c\left\{
        \frac{1}{n} + \frac{1}{(n\lfloor m/2 \rfloor)^{\tilde{\beta} / (\tilde{\beta} + 1)}}\log^3(n\lfloor m/2 \rfloor)
        \right\},
    \end{equation*}
    where $c$ is a universal constant that does not depend on $n$ and $m$.
\end{theorem}

In the sparse regime, 
the convergence rate of deep learning estimators coincides with that of deep nonparametric regression in $2d$ dimensions, 
as established in \cite{SuzukiNitanda2021}. 
The rates do not directly depend on the dimension $d$, 
and thus suggest that deep learning estimators can mitigate the curse of dimensionality. 
For example, set $\beta_1 = \alpha$ and $\beta_2=\dots=\beta_d = (d-1)\alpha$ for some $\alpha>0$.
This represents an anisotropic setting in which the function is non-smooth in the first coordinate, 
while being considerably smoother in the other coordinates.
Then $\tilde{\beta}=\alpha/2$, and the resulting rate does not depend on the dimension $d$.

The bound in Theorem~4.2 exhibits a phase transition with respect to the number of measurements $m$. 
When $m \ll n^{1/\tilde{\beta}}$, the second term dominates and the convergence rate is 
determined by $(n m)^{-\tilde{\beta}/(\tilde{\beta}+1)}$ up to polylogarithmic factors. 
In contrast, when $m \gg n^{1/(\tilde{\beta})}$, the first term $n^{-1}$ dominates, 
and the estimator achieves the parametric rate $1/n$. 
The transition occurs at the critical scaling $m \asymp n^{1/\tilde{\beta}}$, where the two terms are of the same order.

\begin{remark}
    We can easily extend Theorem \ref{theorem:convergence rates} for the affine composition model or the deep composition model in \cite{SuzukiNitanda2021}.
    For example, the affine composition model for covariance functions can be defined as follows:
    \[
    \mathcal{K}_{\mathrm{aff},p,q}^\beta
    :=\left\{ 
    H(As+b, At+b)\mid H\in \mathcal{K}_{p,q}^\beta,\;A\in \Rb^{\tilde{d}\times d},\;b \in \Rb^{\tilde{d}}
    \right\},\quad
    \beta \in \Rb^{\tilde{d}}.
    \]
    Thus, the directions of smoothness need not be aligned with the coordinate directions 
    and can be arbitrarily oriented through affine transformations.
\end{remark}

\subsection{Convergence rates for the tensor product space}\label{sec:TPS}

We consider the setting of \cite{CaiYuan2010} and thus set $\Tcal=[0,1]$.
Let $\Hcal(\kappa)$ be a reproducing kernel Hilbert space (RKHS) with kernel $\kappa:\Tcal\times\Tcal\to\mathbb{R}$. 
We assume that $X$ has sample paths in $\Hcal(\kappa)$ almost surely and that $\mathbb{E}[\|X\|_{\Hcal(\kappa)}^2] < \infty$, 
where $\|\cdot\|_{\Hcal(\kappa)}$ denotes the RKHS norm. 
Then Theorem~1 of \cite{CaiYuan2010} implies that the true covariance function $K^\circ$ lies 
in the tensor product space $\Hcal(\kappa \otimes \kappa)$. 

As a canonical example, 
we consider $\Hcal(\kappa)$ to be the periodic Sobolev space $W_{\rm per}^{\alpha,2}(\Tcal) \subset L^2(\Tcal)$ 
of order $\alpha \in \mathbb{N}$ defined on $\Tcal$.
This setting is also discussed as a representative example in \citet{CaiYuan2010}.
For further details on periodic Sobolev spaces, we refer to Appendix~2.4 of \citet{BerlinetThomas-Agnan2004}.
In this setting, it is well known that 
the eigenfunctions of $\kappa$ are given by the Fourier basis functions (e.g., Section~2 of \cite{Wahba1975}) as follows:
\[
\kappa(s,t) = 1 + \sum_{j=1}^\infty \frac{1}{(2\pi j)^{2\alpha}}
\left\{ 
\psi_j^{\mathrm{(c)}}(s)\psi_j^{\mathrm{(c)}}(t) + \psi_j^{\mathrm{(s)}}(s)\psi_j^{\mathrm{(s)}}(t)
\right\}, \label{eq: eigen expansion}
\]
where $\psi_j^{\mathrm{(c)}}(t) := \sqrt{2}\cos(2\pi jt)$ and $\psi_j^{\mathrm{(s)}}(t) := \sqrt{2}\sin(2\pi jt)$.
For notational convenience, 
we rearrange the eigenfunctions into a single orthonormal sequence $\{\psi_j\}_{j\ge1}$ in $L^2(\Tcal)$ as follows:
\[
\psi_1(t) := 1, \quad
\psi_{2j}(t) := \psi_j^{\rm (c)}(t), \quad
\psi_{2j+1}(t) := \psi_j^{\rm (s)}(t), \quad j \ge 1.
\]
Let $\{\rho_j\}_{j\ge1}$ be the eigenvalues of $\kappa$ associated with $\{\psi_j\}_{j\ge1}$, arranged in non-increasing order.

Since $X$ has sample paths in $\Hcal(\kappa)\subset L^2(\Tcal)$, 
the sample path $X$ and its covariance function $K^\circ$ admit the following expansions:
\begin{equation*}
    X(\cdot) = \sum_{j=1}^\infty Z_j\psi_j(\cdot),\quad Z_j := \int_\Tcal X(t)\psi_j(t)\,dt,
\end{equation*}
and
\begin{equation}
    K^\circ(s, t) = \sum_{j, k \ge 1} c_{jk} \psi_j(s)\psi_k(t), \label{eq: eigen expansion}
\end{equation}
where $c_{jk} := \mathrm{Cov}(Z_j, Z_k)$. 

We consider the following tensor product class of covariance functions defined as the ball of radius $R$ in the RKHS $\Hcal(\kappa \otimes \kappa)$:
\begin{equation*}
    \Kcal_{\rm TP}^\alpha := \{ K \in \Hcal(\kappa \otimes \kappa) \mid \| K \|_{\Hcal(\kappa \otimes \kappa)} \le R \} \label{eq: RKHS hypothesis space},\quad
    \| K \|_{\Hcal(\kappa \otimes \kappa)}^2 := \sum_{j, k \ge 1} \frac{c_{jk}^2}{\rho_j\rho_k}.
\end{equation*}

The following theorem establishes the convergence rates of the deep learning estimator for the class $\Kcal_{\rm TP}^\alpha$.
\begin{theorem} \label{theorem: convergence rates for RKHS}
    Consider the deep neural network model $\Kcal(L, W, S)$ satisfies
    \ba
        &\text{\rm (i) } L \asymp \log^3(n\lfloor m/2 \rfloor),\;\;
        \text{\rm (ii) } W \asymp (n\lfloor m/2 \rfloor)^{1 / (2\alpha+1)}\log^5(n\lfloor m/2 \rfloor),\; \text{ and}\\
        &\text{\rm (iii) } S \asymp (n\lfloor m/2 \rfloor)^{1 / (2\alpha + 1)}\log(n\lfloor m/2 \rfloor).
    \ea
    If $K^\circ \in \Kcal_{\rm TP}^\alpha$ and $\| K^\circ \|_\infty \le B_K$, under Assumption~\ref{assumption:2}, we obtain
    \begin{equation*}
        \Eb[\lVert \what{K}_n - K^\circ \rVert_2^2] 
        \le 
        C\left\{
        \frac{1}{n} + \frac{1}{(n\lfloor m/2 \rfloor)^{2\alpha / (2\alpha + 1)}}\log^{17}(n\lfloor m/2 \rfloor)
        \right\},
    \end{equation*}
    where $c$ is a universal constant independent of $n$ and $m$.
\end{theorem}

For the tensor product space $\Hcal(\kappa \otimes \kappa)$, 
the convergence rate of the deep learning estimator matches that of one-dimensional nonparametric regression, 
even though the covariance function is defined on the two-dimensional domain $[0,1]^2$. 
The same holds for the smoothing spline estimator studied by \citet{CaiYuan2010}.
Moreover, combined with Theorem~6 of \citet{CaiYuan2010}, 
our result shows that the deep learning estimator achieves the minimax rate up to a polylogarithmic factor\footnote{
The polylogarithmic factor $\log^{17}(n\lfloor m/2\rfloor)$ is due to our simple approximation scheme. 
The exponent $17$ is not intrinsic and can likely be reduced.}.

The two approaches, however, rely on fundamentally different mechanisms. 
The smoothing spline estimator achieves the optimal rate 
by explicitly specifying the reproducing kernel associated with the underlying RKHS a priori, 
whereas the deep learning estimator attains the same rate in an adaptive manner, 
without explicit knowledge of the kernel.

\subsection{Comparison with other estimators}\label{sec:comparison}

We compare the deep learning estimator with existing estimators in one-dimensional settings.
Here, we ignore polylogarithmic factors in the convergence rates.
For the local linear smoothing estimator, 
\cite{ZhangWang2016} consider a class of covariance functions satisfying the following condition:
\banum
\frac{\partial^2K(s,t)}{\partial s^2},\;\frac{\partial^2K(s,t)}{\partial s \partial t}\;
\text{ and } \frac{\partial^2K(s,t)}{\partial t^2}\;\text{ are bounded on }[0,1]^2.
\label{eq:condition of Z&W}
\eanum
This class is contained in the anisotropic Besov space $B_{\infty,\infty}^{(2,2)}([0,1]^2)$,
which corresponds to isotropic smoothness in each coordinate.
Here, we note that the LS loss (\ref{eq:intro-ls}) coincide with the equal weight 
per observation (OBS) scheme of \cite{ZhangWang2016}.

For the sparse regime, 
the convergence rates of the local linear smoothing estimator and the deep learning estimator coincide.
For the dense regime, Corollary~4.4 of \citet{ZhangWang2016} shows that 
the local linear smoothing estimator $\Khat_n^{\rm (LLS)}$ with a suitable choice of tuning parameter satisfies
\begin{itemize}
\item $\|\Khat_n^{\rm (LLS)} - K^\circ\|_2 = O_p((nm^2)^{-1/3})$ when $m^2/\sqrt{n} \to 0$, and
\item $\|\Khat_n^{\rm (LLS)} - K^\circ\|_2 = O_p(n^{-1/2})$ when $\sqrt{n} = O(m^2)$.
\end{itemize}
In contrast, Theorem~\ref{theorem:convergence rates} shows that the deep learning estimator $\Khat_n^{\rm (DL)}$ satisfies
\begin{itemize}
\item $\|\Khat_n^{\rm (DL)} - K^\circ\|_2 = O_p((nm)^{-1/3})$ when $m/\sqrt{n} \to 0$, and
\item $\|\Khat_n^{\rm (DL)} - K^\circ\|_2 = O_p(n^{-1/2})$ when $\sqrt{n} = O(m)$.
\end{itemize}
Thus, surprisingly, in the one-dimensional smoothness setting under the dense regime, 
the deep learning estimator is suboptimal, and therefore the local linear smoothing estimator is a suitable choice.

However, the situation can be reversed under a structured setting. 
Indeed, even when the true covariance function $K^\circ$ belongs to $\Hcal(\kappa \otimes \kappa)$ with $\alpha = 2$ 
and also satisfies the condition~(\ref{eq:condition of Z&W}), 
the local linear smoothing estimator yields only the same rates as above.
The local linear smoothing estimator does not exploit the tensor product structure 
and therefore cannot achieve the optimal rate.
By contrast, in this setting, the deep learning estimator, due to its adaptivity, satisfies
\begin{itemize}
\item $\|\Khat_n^{\rm (DL)} - K^\circ\|_2 = O_p((nm)^{-2/5})$ when $m^2/\sqrt{n} \to 0$, and
\item $\|\Khat_n^{\rm (DL)} - K^\circ\|_2 = O_p(n^{-1/2})$ when $\sqrt{n} = O(m^2)$.
\end{itemize}
Therefore, the deep learning estimator is a suitable choice in this setting.

Moreover, while linear estimators such as the local linear smoothing estimator and the smoothing spline estimator lack adaptivity, 
deep learning estimators can adapt to anisotropic smoothness over multidimensional domains.
For example, the suboptimality of linear estimators in classical nonparametric regression is demonstrated in \citet{SuzukiNitanda2021, HayakawaSuzuki2020}.
Similar arguments can be applied to covariance function estimation.


\section{Conclusion}\label{sec:conclusion}

In this paper, we studied nonparametric covariance function estimation for discretely observed functional data 
on a multidimensional domain. 
We established an oracle inequality for a broad class of learning-based estimators 
that applies to both sparse and dense observation regimes in a unified framework. 
As an important application, 
we derived convergence rates for deep learning estimators over several classes of covariance functions.
Our results reveal a phase transition with respect to the number of measurements $m$. 
This transition highlights the interplay between sampling density and intrinsic smoothness in covariance function estimation.

We further demonstrated that the performance of estimators crucially depends on the underlying function class. 
In one-dimensional setting, for a class of smooth functions considered by \citet{ZhangWang2016}, 
deep learning estimators are suboptimal, whereas local linear smoothing estimators achieve faster rates. 
In contrast, for a structured function class, such as the tensor product RKHS studied in \citet{CaiYuan2010}, 
deep learning estimators attain the minimax rate up to polylogarithmic factors, 
while local linear smoothing estimators are suboptimal. 

These findings indicate that learning-based estimators are not always optimal, 
but can effectively exploit structural properties such as anisotropy. 
This leads to a distinctive adaptivity--variance trade-off 
that appears to be specific to covariance function estimation and is not commonly observed in classical nonparametric regression.

As an important direction for future work, 
by combining our framework with the refined analysis of FPCA based on the local linear smoothing estimator by \citet{ZhouEtAl2025}, we may obtain convergence rates for deep learning estimators of the eigenfunctions.

\begin{appendix}

\section{Proof of the key lemma (Lemma~\ref{lemma:modified-Talagrand})}
Let $\mathfrak{S}_m$ denote the symmetric group on $\{1,\dots,m\}$.
Let $\Pi$ be a random permutation taking values in $\mathfrak{S}_m$,
uniformly distributed over all $m!$ permutations.
For a realization $\pi\in\mathfrak{S}_m$, we write
$\pi(j)$ for the image of $j\in\{1,\dots,m\}$ under $\pi$.

First, following the argument of \cite{ClemenconEtAl2008}, note that
\begin{align*}
\frac{1}{m(m-1)} \sum_{j\neq k} f_i(X_{ij},X_{ik})
&=
\frac{1}{m!} \sum_{\pi\in \mathfrak{S}_m}
\frac{1}{\lfloor m/2 \rfloor}
\sum_{j=1}^{\lfloor m/2 \rfloor}
f_i\!\left(
X_{\pi(j)},X_{\pi(\lfloor m/2 \rfloor+j)}
\right) \\
&=
\mathbb{E}_{\Pi}\!\left[
\frac{1}{\lfloor m/2 \rfloor}
\sum_{j=1}^{\lfloor m/2 \rfloor}
f_i\!\left(
X_{\Pi(j)},X_{\Pi(\lfloor m/2 \rfloor+j)}
\right)
\right].
\end{align*}
By Jensen's inequality,
\begin{align*}
S_n
&=
\sup_{f\in\mathcal{F}}
\left|
\frac{m(m-1)}{\lfloor m/2 \rfloor}
\mathbb{E}_{\Pi}\!\left[
\sum_{i=1}^n
\sum_{j=1}^{\lfloor m/2 \rfloor}
f_i\!\left(
X_{i\Pi(j)},X_{i\Pi(\lfloor m/2 \rfloor+j)}
\right)
\right]
\right| \\
&\le
\frac{m(m-1)}{\lfloor m/2 \rfloor}
\mathbb{E}_{\Pi}\!\left[
\sup_{f\in\mathcal{F}}
\left|
\sum_{i=1}^n
\sum_{j=1}^{\lfloor m/2 \rfloor}
f_i\!\left(
X_{i\Pi(j)},X_{i\Pi(\lfloor m/2 \rfloor+j)}
\right)
\right|
\right].
\end{align*}
Define
\begin{align*}
&M_m := \frac{m(m-1)}{\lfloor m/2 \rfloor},\quad
\bar{S}_n
:=
\mathbb{E}_{\Pi}\!\left[
\sup_{f\in\mathcal{F}}
\left|
\sum_{i=1}^n
\sum_{j=1}^{\lfloor m/2 \rfloor}
f_i\!\left(
X_{i\Pi(j)},X_{i\Pi(\lfloor m/2 \rfloor+j)}
\right)
\right|
\right],\\
&\text{and }\;\;S_{n,\pi}
:=
\sup_{f\in\mathcal{F}}
\left|
\sum_{i=1}^n
\sum_{j=1}^{\lfloor m/2 \rfloor}
f_i\!\left(
X_{i\pi(j)},X_{i\pi(\lfloor m/2 \rfloor+j)}
\right)
\right|.
\end{align*}
Then \(S_n \le M_m \bar{S}_n\), and it suffices to derive an inequality for
\(\bar{S}_n\).

We observe that
\begin{align*}
\mathbb{E}_X\!\left[
e^{\lambda(\bar{S}_n-\mathbb{E}_X[\bar{S}_n])}
\right]
&=
\mathbb{E}_X\!\left[
e^{\lambda \mathbb{E}_{\Pi}[S_{n,\Pi}-\mathbb{E}_X S_{n,\Pi}]}
\right] 
\le
\mathbb{E}_{\Pi,X}\!\left[
e^{\lambda (S_{n,\Pi}-\mathbb{E}_X S_{n,\Pi})}
\right].
\end{align*}
By Theorem~3.3.16 of \cite{GineNickl2016}, for any permutation \(\pi\),
\[
\forall \lambda\in[0,2/3],\quad
\mathbb{E}\!\left[
e^{\lambda(S_{n,\pi}/U-\mathbb{E}[S_{n,\pi}/U])}
\right]
\le
\exp\!\left(
\frac{\lambda^2}{2-3\lambda}\,
\frac{v_{n}}{U^2}
\right).
\]
Then, we obtain
\[
\mathbb{E}_X\!\left[
e^{\lambda(\bar{S}_n-\mathbb{E}_X[\bar{S}_n])}
\right]
\le
\exp\!\left(
\frac{\lambda^2}{2-3\lambda}\,
\frac{v_n}{U^2}
\right).
\]
Consequently, for any \(x\ge 0\), with probability at least \(1-e^{-x}\),
\[
\bar{S}_n
\le
\mathbb{E}[\bar{S}_n]
+ \sqrt{2v_n x}
+ \frac{3U}{2}x .
\]
By Young's inequality \(2\sqrt{ab}\le \alpha a + b/\alpha\),
for any $x\ge 0$ and any $\alpha>0$, with probability at least
$1-e^{-x}$,
\begin{align*}
\frac{1}{nm(m-1)}S_n
&\le
\frac{1}{n\lfloor m/2 \rfloor}
\left\{
\mathbb{E}[\bar{S}_n]
+ 2\sqrt{U \mathbb{E}[\tilde{S}_{n}]\,x}
+ \sqrt{2n\lfloor m/2 \rfloor \sigma^2 x}
+ \frac{3U}{2}x
\right\} \\
&\le
(1+\alpha)\sup_{\pi}
\mathbb{E}\!\left[
\frac{\tilde{S}_{n}}{n\lfloor m/2 \rfloor}
\right]
+ \sqrt{\frac{2\sigma^2}{n\lfloor m/2 \rfloor}x}
+ \left(\frac{3}{2}+\frac{1}{\alpha}\right)
\frac{Ux}{n\lfloor m/2 \rfloor}.
\end{align*}
This completes the proof.

\section{Proof of Theorem~\ref{theorem:oracle}}

The proof strategy follows a standard localization analysis (e.g., \citet{BartlettEtAl2005, Koltchinskii2006, BlanchardEtAl2008}) combined with a truncation argument (e.g., \citet{BagirovEtAl2009}), as also employed in \citet{YanEtAl2025}.
The technical lemmas for the proof are given in Appendix~\ref{app:tech-lemmas}.
\medskip

\noindent{\bf Step~1.}
By Lemma~\ref{lemma:CI1} with $C=2$, for every $x\ge 0$ we have, 
with probability at least $1-e^{-x}$,
\begin{align*}
\|\widehat K_n-K^\circ\|_{2}^2
&\le
\frac{2}{nm(m-1)}\sum_{i=1}^n \sum_{j\neq k}
\left\{\widehat K_n(T_{ij},T_{ik}) - K^\circ(T_{ij},T_{ik})\right\}^2\nonumber\\
&\quad
+6400B_K^2 r^\ast + \frac{96B_K^2 x}{n\lfloor m/2\rfloor}.
\end{align*}
This yields
\begin{align}
\label{eq:base}
\mathbb E\bigl[\|\widehat K_n-K^\circ\|_{2}^2\bigr]
&\le
\frac{2}{nm(m-1)}\mathbb E\!\left[
\sum_{i=1}^n \sum_{j\neq k}
\left\{\widehat K_n(T_{ij},T_{ik}) - K^\circ(T_{ij},T_{ik})\right\}^2
\right]\nonumber\\
&\quad+6400B_K^2 r^\ast + \frac{96B_K^2}{n\lfloor m/2\rfloor}.
\end{align}

\medskip
\noindent{\bf Step~2.}
We start from the identity
\begin{align*}
&\sum_{i=1}^n \sum_{j\neq k}\left\{Y_{ij}Y_{ik} - \widehat K_n(T_{ij},T_{ik})\right\}^2\\
&=
\sum_{i=1}^n \sum_{j\neq k}\left\{Y_{ij}Y_{ik} - K^\circ (T_{ij},T_{ik})\right\}^2
+\sum_{i=1}^n \sum_{j\neq k}\left\{K^\circ (T_{ij},T_{ik})-\widehat K_n(T_{ij},T_{ik})\right\}^2\\
&\quad
+2\sum_{i=1}^n \sum_{j\neq k}\left\{Y_{ij}Y_{ik} - K^\circ (T_{ij},T_{ik})\right\}
\left\{K^\circ (T_{ij},T_{ik})-\widehat K_n(T_{ij},T_{ik})\right\}.
\end{align*}
In addition, for any $K\in\Kcal_{nm}$,
\begin{align*}
&\mathbb E\!\left[\sum_{i=1}^n \sum_{j\neq k}\left\{Y_{ij}Y_{ik} - K(T_{ij},T_{ik})\right\}^2\right]\\
&=
\mathbb E\!\left[\sum_{i=1}^n \sum_{j\neq k}\left\{Y_{ij}Y_{ik} - K^\circ (T_{ij},T_{ik})\right\}^2\right]
+
\mathbb E\!\left[\sum_{i=1}^n \sum_{j\neq k}\left\{K^\circ (T_{ij},T_{ik}) - K(T_{ij},T_{ik})\right\}^2\right].
\end{align*}
By the optimality of $\widehat K_n$ over $\Kcal_{nm}$, we therefore obtain
\begin{align*}
&\mathbb E\!\left[
\frac{1}{nm(m-1)}\sum_{i=1}^n \sum_{j\neq k}
\left\{\widehat K_n(T_{ij},T_{ik}) - K^\circ(T_{ij},T_{ik})\right\}^2
\right]\\
&\le
\mathbb E\!\left[
\frac{1}{nm(m-1)}\sum_{i=1}^n \sum_{j\neq k}
\left\{K(T_{ij},T_{ik}) - K^\circ(T_{ij},T_{ik})\right\}^2
\right]\\
&\quad
+2\,\mathbb E\!\left[
\left|
\frac{1}{nm(m-1)}\sum_{i=1}^n \sum_{j\neq k}
\left\{Y_{ij}Y_{ik} - K^\circ (T_{ij},T_{ik})\right\}
\left\{\widehat K_n(T_{ij},T_{ik}) - K^\circ(T_{ij},T_{ik})\right\}
\right|
\right].
\end{align*}
Since $Y_{ij}=X_i(T_{ij})+\epsilon_{ij}$, we can expand
\[
Y_{ij}Y_{ik}
=
\left\{X_i(T_{ij})+\epsilon_{ij}\right\}\left\{X_i(T_{ik})+\epsilon_{ik}\right\}
=
X_i(T_{ij})X_i(T_{ik})
+X_i(T_{ij})\epsilon_{ik}
+X_i(T_{ik})\epsilon_{ij}
+\epsilon_{ij}\epsilon_{ik}.
\]
Thus it suffices to control the three terms
\begin{align*}
U_{nm}
&:=
\left|
\frac{1}{nm(m-1)}\sum_{i=1}^n \sum_{j\neq k}
\left\{X_i(T_{ij})X_i(T_{ik}) - K^\circ (T_{ij},T_{ik})\right\}
\left\{\widehat K_n(T_{ij},T_{ik}) - K^\circ(T_{ij},T_{ik})\right\}
\right|,\\
V_{nm}
&:=
\left|
\frac{1}{nm(m-1)}\sum_{i=1}^n \sum_{j\neq k}
X_i(T_{ij})\epsilon_{ik}
\left\{\widehat K_n(T_{ij},T_{ik}) - K^\circ(T_{ij},T_{ik})\right\}
\right|,\\
W_{nm}
&:=
\left|
\frac{1}{nm(m-1)}\sum_{i=1}^n \sum_{j\neq k}
\epsilon_{ij}\epsilon_{ik}
\left\{\widehat K_n(T_{ij},T_{ik}) - K^\circ(T_{ij},T_{ik})\right\}
\right|.
\end{align*}
Once $\mathbb E[U_{nm}],\mathbb E[V_{nm}],\mathbb E[W_{nm}]$ are bounded, we obtain
\begin{align*}
&\frac{2}{nm(m-1)}\mathbb E\!\left[
\sum_{i=1}^n \sum_{j\neq k}
\left\{\widehat K_n(T_{ij},T_{ik}) - K^\circ(T_{ij},T_{ik})\right\}^2
\right]\\
&\le
2\,\mathbb E\bigl[\|K-K^\circ\|_2^2\bigr]
+4\left\{\mathbb E[U_{nm}]+\mathbb E[V_{nm}]+\mathbb E[W_{nm}]\right\}.
\end{align*}
Combining this with \eqref{eq:base} yields
\begin{align}
\label{eq:pre-total}
\mathbb E\bigl[\|\widehat K_n-K^\circ\|_{2}^2\bigr]
&\le
2\,\mathbb E\bigl[\|K-K^\circ\|_2^2\bigr]
+4\left\{\mathbb E[U_{nm}]+\mathbb E[V_{nm}]+\mathbb E[W_{nm}]\right\}\nonumber\\
&\quad+6400B_K^2 r^\ast + \frac{96B_K^2}{n\lfloor m/2\rfloor}.
\end{align}

\medskip
\noindent{\bf Step~3. (Control of $U_{nm}$).}
Let
\[
e_i(T_{ij},T_{ik})
:=
X_i(T_{ij})X_i(T_{ik}) - K^\circ (T_{ij},T_{ik}).
\]
Introduce a clipping level $\beta_U>0$ and write
{\small
\begin{align*}
U_{nm}
&\le
\Biggl|
\frac{1}{n}\sum_{i=1}^n
\Biggl[
\frac{1}{m(m-1)}\sum_{j\neq k}\{\Ccal_{\beta_U}e_i(T_{ij},T_{ik})\}
\{\widehat K_n(T_{ij},T_{ik}) - K^\circ(T_{ij},T_{ik})\}\\
&\qquad\qquad\quad-
\int \{\Ccal_{\beta_U}e_i(t,t')\}\{\widehat K_n(t,t') - K^\circ(t,t')\}\,dP_T^{\otimes 2}(t,t')
\Biggr]
\Biggr|\\
&\quad+
\left|
\frac{1}{n}\sum_{i=1}^n
\left[
\frac{1}{m(m-1)}\sum_{j\neq k}\{e_i(T_{ij},T_{ik})-\Ccal_{\beta_U}e_i(T_{ij},T_{ik})\}
\{\widehat K_n(T_{ij},T_{ik}) - K^\circ(T_{ij},T_{ik})\}
\right.\right.\\
&\qquad\qquad\qquad\left.\left.
-
\int \{e_i(t,t')-\Ccal_{\beta_U}e_i(t,t')\}\{\widehat K_n(t,t') - K^\circ(t,t')\}\,dP_T^{\otimes 2}(t,t')
\right]
\right|\\
&\quad+
\left|
\frac{1}{n}\sum_{i=1}^n
\int e_i(t,t')\{\widehat K_n(t,t') - K^\circ(t,t')\}\,dP_T^{\otimes 2}(t,t')
\right|\\
&=:U_{nm}^{(\mathrm I)}+U_{nm}^{(\mathrm{II})}+U_{nm}^{(\mathrm{III})}.
\end{align*}
}

\begin{itemize}
\item[(i)] {\it Control of $U_{nm}^{(\mathrm I)}$.}
By Lemma~\ref{lemma:CI2}, for any $C_U>0$,
\begin{align}
\label{eq:U1}
\mathbb E\!\left[U_{nm}^{(\mathrm I)}\right]
\le
\frac{1}{C_U}\,\mathbb E\!\left[\|\widehat K_n-K^\circ\|_{2}^2\right]
+800C_U^2\beta_U^2 r^\ast
+\frac{\beta_U}{n\lfloor m/2\rfloor}\Bigl(C_U\beta_U+44B_K\Bigr).
\end{align}

\item[(ii)] {\it Control of $U_{nm}^{(\mathrm{II})}$.}
Using the pointwise bound $|\widehat K_n-K^\circ|\le 2B_K$, we have
\begin{align*}
\mathbb E\!\left[U_{nm}^{(\mathrm{II})}\right]
&\le
2\,\mathbb E\!\left[
\bigl|
\{e_1(T,T')-\Ccal_{\beta_U}e_1(T,T')\}\{\widehat K_n(T,T') - K^\circ(T,T')\}
\bigr|
\right]\\
&\le
4B_K\,\mathbb E\!\left[|e_1(T,T')|\,\mathbf 1\{|e_1(T,T')|>\beta_U\}\right].
\end{align*}
Moreover, we have
\[
|e_1(T,T')|\le 8B_1^2 \exp\!\left(\frac{|e_1(T,T')|}{8B_1^2}\right)
\;\text{ and }\;
\mathbf 1\{|e_1(T,T')|>\beta_U\}\le \exp\!\left(\frac{|e_1(T,T')|-\beta_U}{8B_1^2}\right).
\]
By the sub-Gaussian assumption, 
\begin{align}
\label{eq:U2}
\mathbb E\!\left[U_{nm}^{(\mathrm{II})}\right]
&\le
32B_KB_1^2\,
\mathbb E\!\left[\exp\!\left(\frac{|e_1(T,T')|}{4B_1^2}\right)\right]\,
\exp\!\left(-\frac{\beta_U}{8B_1^2}\right)\nonumber\\
&\le
64B_KB_1^2\exp\!\left(-\frac{\beta_U}{8B_1^2}\right),
\end{align}
where we used $e^{1/4}\sqrt{2}<2$ in the last step.

\item[(iii)] {\it Control of $U_{nm}^{(\mathrm{III})}$.}
By Young's inequality, for any $\alpha>0$,
\begin{align*}
\mathbb E\!\left[U_{nm}^{(\mathrm{III})}\right]
&=
\mathbb E\!\left[
\left|
\int \left\{\frac{1}{n}\sum_{i=1}^n e_i(t,t')\right\}\left\{\widehat K_n(t,t')-K^\circ(t,t')\right\}\,dP_T^{\otimes 2}(t,t')
\right|
\right]\\
&\le
\frac{\alpha}{2}\,
\mathbb E\!\left[
\int \left\{\frac{1}{n}\sum_{i=1}^n e_i(t,t')\right\}^2\,dP_T^{\otimes 2}(t,t')
\right]
+\frac{1}{2\alpha}\,\mathbb E\!\left[\|\widehat K_n-K^\circ\|_2^2\right].
\end{align*}
By sub-Gaussianity, $\mathbb E[X^4(t)]\le 16B_1^4$, and hence
\[
\mathbb E\!\left[
\int \left\{\frac{1}{n}\sum_{i=1}^n e_i(t,t')\right\}^2\,dP_T^{\otimes 2}(t,t')
\right]
\le \frac{16B_1^4}{n}.
\]
Therefore,
\begin{align}
\label{eq:U3}
\mathbb E\!\left[U_{nm}^{(\mathrm{III})}\right]
\le
\frac{8\alpha B_1^4}{n}
+\frac{1}{2\alpha}\,\mathbb E\!\left[\|\widehat K_n-K^\circ\|_2^2\right].
\end{align}
\end{itemize}
\medskip

\noindent
Taking $\beta_U=8B_1^2\log n$ and combining \eqref{eq:U1}--\eqref{eq:U3}, we obtain
\begin{align}
\mathbb E[U_{nm}]
&\le
\Bigl(\frac{1}{C_U}+\frac{1}{2\alpha}\Bigr)
\mathbb E\!\left[\|\widehat K_n-K^\circ\|_2^2\right]
+800C_U^2(8B_1^2\log n)^2 r^\ast \notag\\
&\quad+\frac{8B_1^2\log n}{n\lfloor m/2\rfloor}
\Bigl(C_U\,8B_1^2\log n+44B_K\Bigr)
+\frac{8(8B_KB_1^2+\alpha B_1^4)}{n} \notag\\
&=:
c_{U,1}\,\mathbb E\!\left[\|\widehat K_n-K^\circ\|_2^2\right]
+c_{U,2}\,\frac{B_KB_1^2+B_1^4}{n}
+c_{U,3}\,\frac{B_1^4\log^2 n+B_KB_1^2\log n}{n\lfloor m/2\rfloor}
\label{eq:U}\\
&\quad
+c_{U,4}\,B_1^4\log^2 n\cdot r^\ast
\notag
\end{align}
where $c_{U,1},\dots,c_{U,4}$ are universal constants.

\medskip
\noindent{\bf Step~4. (Control of $V_{nm}$).}
As in Step~3, we decompose $V_{nm}$ using a clipping level $\beta_V>0$:
\begin{align*}
V_{nm}
&\le
\biggl|
\frac{1}{nm(m-1)}\sum_{i=1}^n\sum_{j\neq k}
\left\{\Ccal_{\beta_V}(\epsilon_{ik}X_i(T_{ij}))
-\mathbb E[\Ccal_{\beta_V}(\epsilon_{ik}X_i(T_{ij}))\mid X_i,T_{ij}]\right\}\\
&\qquad\qquad\qquad\qquad\qquad\times
\left\{\widehat K_n(T_{ij},T_{ik})-K^\circ(T_{ij},T_{ik})\right\}
\biggr|\\
&\quad+
\biggl|
\frac{1}{nm(m-1)}\sum_{i=1}^n\sum_{j\neq k}
\left\{\epsilon_{ik}X_i(T_{ij})
-\Ccal_{\beta_V}(\epsilon_{ik}X_i(T_{ij}))
+\mathbb E[\Ccal_{\beta_V}(\epsilon_{ik}X_i(T_{ij}))\mid X_i,T_{ij}]\right\}\\
&\qquad\qquad\qquad\qquad\qquad\quad\times\left\{\widehat K_n(T_{ij},T_{ik})-K^\circ(T_{ij},T_{ik})\right\}
\biggr|\\
&=:V_{nm}^{(\mathrm I)}+V_{nm}^{(\mathrm{II})}.
\end{align*}
\begin{itemize}
\item[(i)] By Lemma~\ref{lemma:CI3}, for any $C_V>0$,
\begin{align}
\label{eq:V1}
\mathbb E\!\left[V_{nm}^{(\mathrm I)}\right]
\le
\frac{1}{C_V}\,\mathbb E\!\left[\|\widehat K_n-K^\circ\|_2^2\right]
+3200C_V^2\beta_V^2 r^\ast
+\frac{2\beta_V}{n\lfloor m/2\rfloor}\Bigl(C_V\beta_V+26B_K\Bigr).
\end{align}

\item[(ii)] Since $\mathbb E[\epsilon_{ik}X_i(T_{ij})\mid X_i,T_{ij}]=0$, we have
\[
\mathbb E[\Ccal_{\beta_V}(\epsilon_{ik}X_i(T_{ij}))\mid X_i,T_{ij}]
=
-\mathbb E[\epsilon_{ik}X_i(T_{ij})-\Ccal_{\beta_V}(\epsilon_{ik}X_i(T_{ij}))\mid X_i,T_{ij}].
\]
Therefore, by Jensen's inequality and the triangle inequality,
\begin{align*}
&\Eb\Big[
\bigl|\epsilon_{ik}X_i(T_{ij})
-\Ccal_{\beta_V}(\epsilon_{ik}X_i(T_{ij}))
+\mathbb E[\Ccal_{\beta_V}(\epsilon_{ik}X_i(T_{ij}))\mid X_i,T_{ij}]\bigr|
\Big]\\
&\le
2\,\Eb\Big[|\epsilon_{ik}X_i(T_{ij})|\,\mathbf 1\{|\epsilon_{ik}X_i(T_{ij})|>\beta_V\}\Big].
\end{align*}
Consequently,
\begin{align}
\label{eq:V2}
\Eb\left[V_{nm}^{(\mathrm{II})}\right]
&\le
4B_K\,\Eb\Big[|\epsilon_{ik}X_i(T_{ij})|\,\mathbf 1\{|\epsilon_{ik}X_i(T_{ij})|>\beta_V\}\Big]\notag\\
&\le
32B_KB_1B_2\exp\!\left(-\frac{\beta_V}{2B_1B_2}\right),
\end{align}
where the last inequality follows from the sub-Gaussian assumption.
\end{itemize}
Taking $\beta_V=2B_1B_2\log(n\lfloor m/2\rfloor)$ and combining \eqref{eq:V1}--\eqref{eq:V2} yields
\begin{align}
\mathbb E[V_{nm}]
&\le
\frac{1}{C_V}\,\mathbb E\!\left[\|\widehat K_n-K^\circ\|_2^2\right]
+3200C_V^2\bigl(2B_1B_2\log(n\lfloor m/2\rfloor)\bigr)^2 r^\ast \notag\\
&\quad+
\frac{2B_1B_2\left\{11\log(n\lfloor m/2\rfloor)\bigl(C_V\,2B_1B_2\log(n\lfloor m/2\rfloor)+26B_K\bigr)+16B_K\right\}}{n\lfloor m/2\rfloor}\notag\\
&=:
c_{V,1}\,\mathbb E\!\left[\|\widehat K_n-K^\circ\|_2^2\right]\notag\\
&\quad+c_{V,2}B_1B_2\,\frac{B_1B_2\log^2(n\lfloor m/2\rfloor)+B_K\log(n\lfloor m/2\rfloor)+B_K}{n\lfloor m/2\rfloor}
\label{eq:V}\\
&\quad+c_{V,3}\,B_1^2B_2^2\log^2(n\lfloor m/2\rfloor)\, r^\ast\notag,
\end{align}
where $c_{V,1},c_{V,2},c_{V,3}$ are universal constants.

\medskip
\noindent{\bf Step~5. (Control of $W_{nm}$).}
With a clipping level $\beta_W>0$, decompose
\begin{align*}
W_{nm}
&\le
\left|
\frac{1}{nm(m-1)}\sum_{i=1}^n\sum_{j\neq k}
\left\{\Ccal_{\beta_W}(\epsilon_{ij}\epsilon_{ik})-\mathbb E[\Ccal_{\beta_W}(\epsilon_{ij}\epsilon_{ik})]\right\}
\left\{\widehat K_n(T_{ij},T_{ik})-K^\circ(T_{ij},T_{ik})\right\}
\right|\\
&\quad+
\Bigg|
\frac{1}{nm(m-1)}\sum_{i=1}^n\sum_{j\neq k}
\left\{\epsilon_{ij}\epsilon_{ik}-\Ccal_{\beta_W}(\epsilon_{ij}\epsilon_{ik})+\mathbb E[\Ccal_{\beta_W}(\epsilon_{ij}\epsilon_{ik})]\right\}\\
&\quad\qquad\qquad\qquad\qquad\qquad\times\left\{\widehat K_n(T_{ij},T_{ik})-K^\circ(T_{ij},T_{ik})\right\}
\Bigg|\\
&=:W_{nm}^{(\mathrm I)}+W_{nm}^{(\mathrm{II})}.
\end{align*}
\begin{itemize}
\item[(i)] By Lemma~\ref{lemma:CI4}, for any $C_W>0$,
\begin{align}
\label{eq:W1}
\mathbb E\!\left[W_{nm}^{(\mathrm I)}\right]
\le
\frac{1}{C_W}\,\mathbb E\!\left[\|\widehat K_n-K^\circ\|_2^2\right]
+3200C_W^2\beta_W^2 r^\ast
+\frac{2\beta_W}{n\lfloor m/2\rfloor}\Bigl(C_W\beta_W+26B_K\Bigr).
\end{align}

\item[(ii)] Since $\mathbb E[\epsilon_{ij}\epsilon_{ik}]=0$, the same argument as in Step~4 gives
\[
\mathbb E\!\left[
\bigl|\epsilon_{ij}\epsilon_{ik}-\Ccal_{\beta_W}(\epsilon_{ij}\epsilon_{ik})+\mathbb E[\Ccal_{\beta_W}(\epsilon_{ij}\epsilon_{ik})]\bigr|
\right]
\le
2\,\mathbb E\!\left[|\epsilon_{ij}\epsilon_{ik}|\,\mathbf 1\{|\epsilon_{ij}\epsilon_{ik}|>\beta_W\}\right],
\]
and hence
\begin{align}
\label{eq:W2}
\mathbb E\!\left[W_{nm}^{(\mathrm{II})}\right]
\le
4B_K\,\mathbb E\!\left[|\epsilon_{ij}\epsilon_{ik}|\,\mathbf 1\{|\epsilon_{ij}\epsilon_{ik}|>\beta_W\}\right]
\le
32B_KB_2^2\exp\!\left(-\frac{\beta_W}{2B_2^2}\right).
\end{align}
\end{itemize}
Taking $\beta_W=2B_2^2\log(n\lfloor m/2\rfloor)$ and combining \eqref{eq:W1}--\eqref{eq:W2} yields
\begin{align}
\mathbb E[W_{nm}]
&\le
\frac{1}{C_W}\,\mathbb E\!\left[\|\widehat K_n-K^\circ\|_2^2\right]
+3200C_W^2\bigl(2B_2^2\log(n\lfloor m/2\rfloor)\bigr)^2 r^\ast \notag\\
&\quad+
\frac{22B_2^2\log(n\lfloor m/2\rfloor)\bigl(C_W\,2B_2^2\log(n\lfloor m/2\rfloor)+26B_K\bigr)+32B_KB_2^2}{n\lfloor m/2\rfloor}\notag\\
&=:
c_{W,1}\,\mathbb E\!\left[\|\widehat K_n-K^\circ\|_2^2\right]
+c_{W,2}B_2^2\,\frac{B_2^2\log^2(n\lfloor m/2\rfloor)+B_K\log(n\lfloor m/2\rfloor)+B_K}{n\lfloor m/2\rfloor}
\label{eq:W}\\
&\quad+c_{W,3}\,B_2^4\log^2(n\lfloor m/2\rfloor)\, r^\ast\notag,
\end{align}
where $c_{W,1},c_{W,2},c_{W,3}$ are universal constants.

\medskip
\noindent{\bf Step~6. (Aggregation).}
Substituting \eqref{eq:U}, \eqref{eq:V}, and \eqref{eq:W} into \eqref{eq:pre-total} and collecting terms, we arrive at
\begin{align*}
&\mathbb E\bigl[\|\widehat K_n-K^\circ\|_{2}^2\bigr]\\
&\le
Const.\times\Biggl[
\mathbb E\bigl[\|K-K^\circ\|_2^2\bigr]
+\frac{B_KB_1^2+B_1^4}{n}\\
&\;\qquad\qquad\qquad
+\frac{1}{n\lfloor m/2\rfloor}\Biggl\{
(B_1^2B_2^2+B_2^4)\log^2\!\bigl(n\lfloor m/2\rfloor\bigr)
+(B_1B_2B_K+B_KB_2^2)\log\!\bigl(n\lfloor m/2\rfloor\bigr)\\
&\;\qquad\qquad\qquad\qquad\qquad\quad
+B_1^4\log^2 n+B_KB_1^2\log n
+B_1B_2B_K+B_KB_2^2+B_K^2
\Biggr\}\\
&\;\qquad\qquad\qquad
+r^\ast\Biggl\{
B_1^4\log^2 n
+(B_1^2B_2^2+B_2^4)\log^2\!\bigl(n\lfloor m/2\rfloor\bigr)
+B_K^2
\Biggr\}
\Biggr].
\end{align*}
In particular, absorbing the fixed parameters $B_1,B_2,B_K$ into the overall constant, 
we obtain the oracle inequalty
\begin{align*}
\mathbb E\bigl[\|\widehat K_n-K^\circ\|_{2}^2\bigr]
\le
Const.\times
\left\{
\inf_{K\in\Kcal}\mathbb E\bigl[\|K-K^\circ\|_2^2\bigr]
+\frac{1}{n}
+\log^2\!\bigl(n\lfloor m/2\rfloor\bigr)\biggl(\frac{1}{n\lfloor m/2\rfloor}+r^\ast\biggr)
\right\}.
\end{align*}

\section{Proof of Corollary~\ref{corollary:oracle-without-fixed-point}}

\begin{lemma} [Duddley's entropy integral, Theorem 2 in \cite{Srebo2010}] \label{lem: Duddley's entropy integral}
    Let $\Fcal$ be a function class and suppose that $\sup_{f \in \Fcal} \sum_{i=1}^n (f(X_i) - f^\circ(X_i))^2 / n \le c_1^2$ with given $\Xcal = \{X_i, 1 \le i \le n\}$.
    Then we have
    \begin{equation*}
        \Eb\left[ \sup_{f \in \Fcal} \left\lvert \frac{1}{n} \sum_{i=1}^n \sigma_i (f(X_i) - f^\circ(X_i)) \right\rvert \right] \le \inf_{0 \le t \le c_1} \left\{ 4t + \frac{12}{\sqrt{n}} \int_t^{c_1} \sqrt{\log\mathcal{N}(\varepsilon, \Fcal\vert_\Xcal, L^2)}\; d\varepsilon \right\},
    \end{equation*}
    where the expectation is taken with respect to $\sigma_i$.
\end{lemma}

\begin{lemma}[Corollary 2.2 in \cite{BartlettEtAl2005}] \label{lem: Corollary 2.2 in Bartlett et al. (2005)}
Let $\mathcal{F}$ be a function class and suppose that $\lVert f\rVert_\infty \le 1$ for any $f\in\mathcal{F}$.
Suppose that positive numbers $r$ and $t$ sattisfy
\begin{equation*}
r \ge 10\,\mathbb{E}\!\left[\sup_{f\in\mathcal{F}(r)}
\left\lvert \frac1n\sum_{i=1}^n \sigma_i\bigl(f(X_i)-f^\circ(X_i)\bigr)\right\rvert\right]
+\frac{11t}{n},
\end{equation*}
where $\{X_i, 1\le i \le n\}$ are i.i.d. random variables.
Then, with probability at least $1-e^{-t}$, one holds:
\begin{equation*}
\mathcal{F}(r)\subset \left\{f\in\mathcal{F}:\ \frac1n\sum_{i=1}^n \bigl(f(X_i)-f^\circ(X_i)\bigr)^2 \le 2r\right\}.
\end{equation*}
\end{lemma}

\begin{proof}{(Proof of Corollary~\ref{corollary:oracle-without-fixed-point})}
Without loss of generality, we may assume that $\lVert K^\circ \rVert \leq 1$ and $\lVert K \rVert \leq 1$ for all $K \in \mathcal{K}$.
By theorem \ref{theorem:oracle}, we have
\begin{equation}
    \mathbb{E}\left[\lVert \hat{K}_n - K^\circ \rVert_2^2\right]
    \leq
    c \times
    \left\{
        \inf_{K \in \mathcal{K}} \mathbb{E}\left[\lVert K - K^\circ \rVert_2^2\right]
        + \frac{1}{n}
        + \log^2\left(n \lfloor m/2 \rfloor\right)
        \left(
            \frac{1}{n \lfloor m/2 \rfloor}
            + r^*
        \right)
    \right\}.
    \label{eq:oracle_ineq_vc_only}
\end{equation}
Define
\begin{align*}
    r_1
    :=
    \sup \Bigg\{
        r > 0 :&
        \
        10 \mathbb{E}\left[
            \sup_{K \in \mathcal{K}(r)}
            \left\lvert
            \frac{1}{n \lfloor m/2 \rfloor}
            \sum_{i=1}^n \sum_{j=1}^{\lfloor m/2 \rfloor}
            \sigma_{ij}
            \left\{
                K(T_{ij_1}, T_{ij_2})
                -
                K^\circ(T_{ij_1}, T_{ij_2})
            \right\}
            \right\rvert
        \right] \\
        &+
        \frac{11 \log\left(n \lfloor m/2 \rfloor\right)}
        {n \lfloor m/2 \rfloor}
        \leq r
    \Bigg\}.
\end{align*}
Since
\begin{equation*}
    \mathbb{E}\left[
        \sup_{K \in \mathcal{K}(r)}
        \left\lvert
        \frac{1}{n \lfloor m/2 \rfloor}
        \sum_{i=1}^n \sum_{j=1}^{\lfloor m/2 \rfloor}
        \sigma_{ij}
        \left\{
            K(T_{ij_1}, T_{ij_2})
            -
            K^\circ(T_{ij_1}, T_{ij_2})
        \right\}
        \right\rvert
    \right]
\end{equation*}
is monotone nondecreasing in $r$, it follows that
\begin{align}
    r_1
    =&
    10 \mathbb{E}\left[
        \sup_{K \in \mathcal{K}(r_1)}
        \left\lvert
        \frac{1}{n \lfloor m/2 \rfloor}
        \sum_{i=1}^n \sum_{j=1}^{\lfloor m/2 \rfloor}
        \sigma_{ij}
        \left\{
            K(T_{ij_1}, T_{ij_2})
            -
            K^\circ(T_{ij_1}, T_{ij_2})
        \right\}
        \right\rvert
    \right] \label{eq:def_r1_vc_clean} \\
    &+
    \frac{11 \log\left(n \lfloor m/2 \rfloor\right)}
    {n \lfloor m/2 \rfloor}. \nonumber 
\end{align}
Moreover, we have $r^* \leq r_1$.

For any $r \geq r_1$, let
\begin{equation*}
    \hat{\mathcal{K}}(2r)
    :=
    \left\{
        K \in \mathcal{K}
        :
        \frac{1}{n \lfloor m/2 \rfloor}
        \sum_{i=1}^n \sum_{j=1}^{\lfloor m/2 \rfloor}
        \left\{
            K(T_{ij_1}, T_{ij_2})
            -
            K^\circ(T_{ij_1}, T_{ij_2})
        \right\}^2
        \leq 2r
    \right\}.
\end{equation*}
Then the local Rademacher complexity $\mathcal{R}_{nm}(\Kcal(r))$ satisfies
\begin{align*}
    \mathcal{R}_{nm}(\mathcal{K}(r))
    :=&
    \mathbb{E}\left[
        \sup_{K \in \mathcal{K}(r)}
        \left\lvert
        \frac{1}{n \lfloor m/2 \rfloor}
        \sum_{i=1}^n \sum_{j=1}^{\lfloor m/2 \rfloor}
        \sigma_{ij}
        \left\{
            K(T_{ij_1}, T_{ij_2})
            -
            K^\circ(T_{ij_1}, T_{ij_2})
        \right\}
        \right\rvert
    \right]
    \\
    \leq&
    \mathbb{E}\left[
        \sup_{K \in \hat{\mathcal{K}}(2r)}
        \left\lvert
        \frac{1}{n \lfloor m/2 \rfloor}
        \sum_{i=1}^n \sum_{j=1}^{\lfloor m/2 \rfloor}
        \sigma_{ij}
        \left\{
            K(T_{ij_1}, T_{ij_2})
            -
            K^\circ(T_{ij_1}, T_{ij_2})
        \right\}
        \right\rvert
    \right]
    \\
    +&
    \mathbb{E}\left[
        \sup_{K \in \mathcal{K}(r) \setminus \hat{\mathcal{K}}(2r)}
        \left\lvert
        \frac{1}{n \lfloor m/2 \rfloor}
        \sum_{i=1}^n \sum_{j=1}^{\lfloor m/2 \rfloor}
        \sigma_{ij}
        \left\{
            K(T_{ij_1}, T_{ij_2})
            -
            K^\circ(T_{ij_1}, T_{ij_2})
        \right\}
        \right\rvert
    \right].
\end{align*}

By Lemma~\ref{lem: Duddley's entropy integral}, the first term is bounded by
\begin{align*}
    &
    \mathbb{E}\left[
        \sup_{K \in \hat{\mathcal{K}}(2r)}
        \left\lvert
        \frac{1}{n \lfloor m/2 \rfloor}
        \sum_{i=1}^n \sum_{j=1}^{\lfloor m/2 \rfloor}
        \sigma_{ij}
        \left\{
            K(T_{ij_1}, T_{ij_2})
            -
            K^\circ(T_{ij_1}, T_{ij_2})
        \right\}
        \right\rvert
    \right]
    \\
    \leq&
    \frac{4}{n \lfloor m/2 \rfloor}
    +
    \frac{12}{\sqrt{n \lfloor m/2 \rfloor}}
    \int_{1 / \left(n \lfloor m/2 \rfloor\right)}^{\sqrt{2r}}
    \sqrt{
        \log
        \mathcal{N}\left(
            \varepsilon,
            \mathcal{K}\vert_{D_n},
            L_2
        \right)
    }
    \, d\varepsilon
    \\
    \leq&
    \frac{4}{n \lfloor m/2 \rfloor}
    +
    \frac{12\sqrt{2r}}{\sqrt{n \lfloor m/2 \rfloor}}
    \sqrt{
        \log
        \mathcal{N}\left(
            \frac{1}{n \lfloor m/2 \rfloor},
            \mathcal{K}\vert_{D_n},
            L_2
        \right)
    }.
\end{align*}

For the second term, we have
\begin{align*}
    &
    \mathbb{E}\left[
        \sup_{K \in \mathcal{K}(r) \setminus \hat{\mathcal{K}}(2r)}
        \left\lvert
        \frac{1}{n \lfloor m/2 \rfloor}
        \sum_{i=1}^n \sum_{j=1}^{\lfloor m/2 \rfloor}
        \sigma_{ij}
        \left\{
            K(T_{ij_1}, T_{ij_2})
            -
            K^\circ(T_{ij_1}, T_{ij_2})
        \right\}
        \right\rvert
    \right]
    \\
    \leq&
    2 \mathbb{P}\left(
        \exists K \in \mathcal{K}(r) :
        \frac{1}{n \lfloor m/2 \rfloor}
        \sum_{i=1}^n \sum_{j=1}^{\lfloor m/2 \rfloor}
        \left\{
            K(T_{ij_1}, T_{ij_2})
            -
            K^\circ(T_{ij_1}, T_{ij_2})
        \right\}^2
        > 2r
    \right)
    \\
    \leq&
    \frac{2}{n \lfloor m/2 \rfloor},
\end{align*}
where the last inequality follows from Lemma~\ref{lem: Corollary 2.2 in Bartlett et al. (2005)} together with \eqref{eq:def_r1_vc_clean}. Hence,
\begin{equation*}
    \mathcal{R}_{nm}(\mathcal{K}(r))
    \leq
    \frac{6}{n \lfloor m/2 \rfloor}
    +
    \frac{12\sqrt{2r}}{\sqrt{n \lfloor m/2 \rfloor}}
    \sqrt{
        \log
        \mathcal{N}\left(
            \frac{1}{n \lfloor m/2 \rfloor},
            \mathcal{K}\vert_{D_n},
            L_2
        \right)
    }.
\end{equation*}

Moreover, we obtain
\begin{align*}
    r_1
    \leq&
    \frac{60}{n \lfloor m/2 \rfloor}
    +
    \frac{120\sqrt{2r_1}}{\sqrt{n \lfloor m/2 \rfloor}}
    \mathbb{E}\left[
        \sqrt{
            \log
            \mathcal{N}\left(
                \frac{1}{n \lfloor m/2 \rfloor},
                \mathcal{K}\vert_{D_n},
                L_2
            \right)
        }
    \right]
    +
    \frac{11\log\left(n \lfloor m/2 \rfloor\right)}
    {n \lfloor m/2 \rfloor}
    \\
    \leq&
    \frac{60}{n \lfloor m/2 \rfloor}
    +
    \frac{r_1}{2}
    +
    \frac{120^2}{n \lfloor m/2 \rfloor}
    \mathbb{E}\left[
        \log
        \mathcal{N}\left(
            \frac{1}{n \lfloor m/2 \rfloor},
            \mathcal{K}\vert_{D_n},
            L_2
        \right)
    \right]
    +
    \frac{11\log\left(n \lfloor m/2 \rfloor\right)}
    {n \lfloor m/2 \rfloor},
\end{align*}
where the second inequality follows from the arithmetic--geometric mean inequality. Therefore, there exists a constant $C>0$ such that
\begin{equation*}
    r_1
    \leq
    \frac{C}{n \lfloor m/2 \rfloor}
    \left(
        \mathbb{E}\left[
            \log
            \mathcal{N}\left(
                \frac{1}{n \lfloor m/2 \rfloor},
                \mathcal{K}\vert_{D_n},
                L_2
            \right)
        \right]
        +
        \log\left(n \lfloor m/2 \rfloor\right)
    \right).
\end{equation*}

By Theorem 9.4 of \cite{Gyorfi_etal_2002} and the relation between covering numbers and packing numbers,
\begin{align*}
    \log
    \mathcal{N}\left(
        \frac{1}{n \lfloor m/2 \rfloor},
        \mathcal{K}\vert_{D_n},
        L_2
    \right)
    \leq&
    \mathrm{VCdim}(\mathcal{K})
    \log\left(
        2e \left(n \lfloor m/2 \rfloor\right)^2
        \log\left(3e \left(n \lfloor m/2 \rfloor\right)^2\right)
    \right) \\
    &+
    \log 3.
\end{align*}
Since $n \lfloor m/2 \rfloor \geq 1$, the statement of the corollary follows.
\end{proof}

\section{Proof of Theorem~\ref{theorem:convergence rates}}

\begin{lemma}[Proposition 2 in \cite{SuzukiNitanda2021} (the case $r=2$)] \label{lem: approximation error for anisotropic Besov}
Let $0 < p,q \leq \infty$ and $\beta \in \mathbb{R}_{+}^{2d}$ satisfy $\tilde{\beta} / 2 > \left(1/p - 1/2\right)_{+}$.
Assume further that $l \in \mathbb{N}$ satisfies $0 < \bar{\beta} < \min\left(l,\, l - 1 + 1/p\right)$.
Set
\begin{equation*}
    \delta = 2\left(1/p - 1/2\right)_{+},
    \qquad
    \nu = \frac{\tilde{\beta}/2 - \delta}{2\delta},
    \qquad
    W_0 := 12dl(l+2)+4d.
\end{equation*}
Then, for any $f \in B^{\beta}_{p,q}([0,1]^{2d})$ with $\|f\|_{B^{\beta}_{p,q}}\le 1$ and for any $N \in \mathbb{N}$, 
there exists a neural network $\tilde{f} \in \Fcal(L, W, S)$ satisfying
\begin{align*}
    &\text{\rm (i) }
    L = 3 + 2\log_2\left(\frac{3(2d \vee l)}{\epsilon\, c(d,l)}\right) + 5\log_2(2d \vee l),\quad
    \text{\rm (ii) }
    W = N W_0,\\
    &\text{\rm (iii) }
    S = \left((L-1)W_0^2 + 1\right)N, 
\end{align*}
and
\begin{equation*}
    \lVert \tilde{f} - f \rVert_{2} \leq N^{-\tilde{\beta}},
\end{equation*}
where $\epsilon = N^{-\tilde{\beta}}\log(N)^{-1}$, and $c(d,l)$ is a constant depending only on $d$ and $l$.
\end{lemma}

\begin{proof}[Proof of Theorem~\ref{theorem:convergence rates}]
    In this proof, $Const.$ denotes a positive constant 
    whose value may change from line to line.
    By Theorem 6 in \cite{Bartlett_etal_2019}, 
    the VC dimension of the class of DNNs $\Fcal(L, W, S)$ is bounded as
    \begin{equation*}
        \mathrm{VCdim} \le Const.\, LS\log(S),
    \end{equation*}
    where $Const.$ is independent of $L$ and $S$.
    Substituting this bound into the result of Corollary~\ref{corollary:oracle-without-fixed-point}, 
    we obtain
    \begin{equation}
        \Eb[\| \hat{K}_n - K^\circ \|_2^2] 
        \le 
        Const.\left\{ \inf_{K \in \Kcal(L, W, S, B)} \Eb[\| K - K^\circ \|_2^2] + \frac{1}{n} + \frac{LS\log(S) \log(n \lfloor m/2 \rfloor)}{n \lfloor m/2 \rfloor} \right\}, \label{eq: oracle for NN}
    \end{equation}
    where $Const.$ is independent of $n$ and $m$.
    
    By Lemma~\ref{lem: approximation error for anisotropic Besov}, 
    taking $N = (n\lfloor m/2 \rfloor)^{1 / (\tilde{\beta} + 1)}$ 
    and the corresponding $L$, $W$, and $S$, 
    we have
    \begin{align*}
        \inf_{K\in \Kcal(L, W, S)}\| K - K^\circ \|_2^2 
        &= \inf_{h\in \Fcal(L, W, S)}\left\| \frac{1}{2} \left\{ h(\cdot, \star) + h(\star, \cdot)\right\}- K^\circ \right\|_2^2 \\
        &\le \frac{1}{2} \inf_{h\in \Fcal(L, W, S)}\| h(\cdot, \star) - K^\circ \|_2^2 + \frac{1}{2} \inf_{h\in \Fcal(L, W, S)}\| h(\star, \cdot) - K^\circ \|_2^2\\
        &=  \inf_{h\in \Fcal(L, W, S)}\| h(\cdot, \star) - K^\circ \|_2^2\\
        &\le (n \lfloor m / 2 \rfloor)^{-\tilde{\beta} / (\tilde{\beta} + 1)}.
    \end{align*}
    By substituting this bound and the values of $L$ and $S$ corresponding to $N = (n\lfloor m/2 \rfloor)^{1 / (\tilde{\beta} + 1)}$ into \eqref{eq: oracle for NN}, 
    we obtain the desired bound.
\end{proof}

\section{Proof of Theorem~\ref{theorem: convergence rates for RKHS}}
\begin{lemma}[Lemma 1.2 in \cite{chenLiu2022}] \label{lem:approximation error for cos}
    Let $D$ be any bounded subset of $\Rb^d$ and $C_D := \sup_{x \in [-1, 1]^d} \sup_{y \in D} |  x^\top y |$.
    For $\epsilon \in (0,1/2)$, 
    There exists a deep ReLU network $\tilde{X}_\epsilon$ with $O(\log(C_DM)\log^2(1/\epsilon))$ layers 
    and $O(\log(C_DM) \log^3(1/\epsilon)))$ nodes such that
    \begin{equation*}
        | \tilde{X}_\epsilon - \cos(2\pi w^\top x) | \le \epsilon, \quad\text{for } w \in [-M, M]^d,\quad x \in D.
    \end{equation*}
\end{lemma}

\begin{lemma} \label{lem: approximation error for products of 1d Fourier basis}
Let $\{\psi_j\}_{j\geq 1}$ be the Fourier basis of $L^2([0,1])$ given by
\[
\psi_1(x) := 1,
\qquad
\psi_{2j}(t) := \sqrt{2}\cos(2\pi j t),
\qquad
\psi_{2j+1}(t) := \sqrt{2}\sin(2\pi j t),
\quad j\in\mathbb{N}.
\]
For any $j,k\in\mathbb{N}$, let $\Lambda_{j,k} = \log\left(2(j\vee k)\right).$
There exists $c_{L,W}>0$, independent of $j$ and $k$, such that the following holds.
For any $\epsilon$ and integers $L$ and $W$ satisfying $L\le  c_{L,W}\Lambda_{j,k}\log^2(1/\epsilon)$ and width $W\le c_{L,W}\Lambda_{j,k}\log^3(1/\epsilon)$, 
one can construct a deep ReLU network $\widetilde \Psi_{jk}:[0,1]^2\to\mathbb{R}$ 
with depth $L$ and width $W$ satisfying
\begin{align*}
\left\lVert
\psi_j(s)\psi_k(t)-\widetilde \Psi_{jk}(s,t)
\right\rVert_{\infty}
\leq \epsilon.
\end{align*}
\end{lemma}

\begin{proof}
We first note that, by definition of the Fourier basis, each $\psi_j$ is either the constant function $1$, 
or a constant multiple of $\cos(2\pi l t)$, or a constant multiple of $\sin(2\pi l t)$ for some integer $l \geq 1$.
Moerover, by construction, $\psi_1$ has frequency $0$, and both $\psi_{2j}$ and $\psi_{2j+1}$ have frequency $j$. 
Thus, the frequency associated with $\psi_j$ is given by $l_j := \left\lfloor j/2\right\rfloor.$
Hence, for any $j,k\geq 1$,
\[
l_j \vee l_k 
= \left\lfloor \frac{j}{2} \right\rfloor \vee \left\lfloor \frac{k}{2} \right\rfloor
\le \frac{j}{2} \vee \frac{k}{2}
= \frac{1}{2}(j \vee k)
\le j \vee k.
\]

If either $\psi_j$ or $\psi_k$ is constant, 
then $\psi_j(s)\psi_k(t)$ reduces to a one-dimensional sine or cosine function 
with frequency at most $j\vee k$, and the result follows from Lemma~\ref{lem:approximation error for cos}.
We therefore consider the nonconstant case.
In this case, up to a multiplicative constant, 
$\psi_j(s)\psi_k(t)$ is a product of sine and/or cosine functions with frequencies $l_j$ and $l_k$. 
By standard trigonometric identities, 
it can be written as a linear combination of at most two functions of the form
\[
\sin\left(2\pi \left(\pm l_j s \pm l_k t\right)\right)
\quad\text{or}\quad
\cos\left(2\pi \left(\pm l_j s \pm l_k t\right)\right).
\]

Let $u := (s,t)\in[0,1]^2.$
Each function above is of the form
\[
\sin\left(2\pi \xi^\top u\right)
\quad\text{or}\quad
\cos\left(2\pi \xi^\top u\right),
\]
where $\xi \in \{\pm l_j\}\times \{\pm l_k\}.$
Since $l_j \vee l_k \leq j\vee k$, we have $\|\xi\|_\infty \leq j\vee k.$

We first consider $\cos(2\pi \xi^\top u)$.
For a deep ReLU network $F$, let $L(F)$ and $W(F)$ denote its depth and width, respectively.
Applying Lemma~\ref{lem:approximation error for cos} with $d=2$, $D=[0,1]^2$, and $M=j\vee k$, 
it follows that, for any $\epsilon\in(0,1/2)$, there exists a deep ReLU network $F_\epsilon^{(\cos)}$ such that
\[
\|F_\epsilon^{(\cos)}(\cdot) - \cos(2\pi \xi^\top \cdot)\|_{\infty}
\le \varepsilon,
\]
and
\[
L\left(F_\epsilon^{(\cos)}\right)
\le c_L \Lambda_{j,k}\log^2(1/\epsilon),\quad
W\left(F_\epsilon^{(\cos)}\right)
\le c_W \Lambda_{j,k}\log^3(1/\epsilon),
\]
where $c_L,c_W>0$ are constants independent of $\epsilon$, $j$, and $k$.
The same conclusion holds for $\sin(2\pi \xi^\top u)$, since
\[
\sin(2\pi \xi^\top u)
=
\cos\!\left(2\pi \xi^\top u - \frac{\pi}{2}\right).
\]
Thus, the statement of this lemma is proved with $c_{L, W} := c_L \vee c_W$.
\end{proof}

\begin{proof}[Proof of Theorem~\ref{theorem: convergence rates for RKHS}]
    In this proof, $\mathrm{Const.}$ denotes a positive constant whose value may change from line to line.
    For $M\in\mathbb{N}$, define
    $
    I_M
    :=
    \left\{
    (j,k)\in\mathbb{N}^2: jk \leq M
    \right\},
    $
    and define the truncation of $K^\circ$ at level $M$ by
    \ba
        K_M^\circ(s,t) := \sum_{(j,k)\in I_M} c_{jk}\psi_j(s)\psi_k(t).
    \ea
    Since $\{\psi_j\psi_k\}$ are orthonormal basis of $L^2(\Tcal ^2)$, we have
    \ba
        \| K_M^\circ - K^\circ \|^2_2 = \sum_{(j,k)\not\in I_M} c_{jk}^2.
    \ea
    By the definition of $\Kcal_{\rm TP}^\alpha$ and $\rho_j = (2\pi \lfloor j/2 \rfloor)^{-2\alpha}$ for $j \ge 2$,
    \begin{equation*}
        \sum_{(j,k)\notin\Lambda_M} c_{jk}^2
        \le (\sup_{jk > M} \rho_j\rho_k) \sum_{j, k \ge 1} \frac{c_{jk}^2}{\rho_j\rho_k} 
        \le R^2\sup_{jk > M} \rho_j\rho_k
        \le Const.\times \frac{R^2}{M^{2\alpha}},
    \end{equation*}
    where $Const.$ is independent of $M$.
    Thus, we obtain
    \begin{equation}
        \| K_M^\circ - K^\circ  \|^2_2 \lesssim \frac{R^2}{M^{2\alpha}}.
        \label{eq:truncation}
    \end{equation}

    By Lemma~\ref{lem: approximation error for products of 1d Fourier basis},
    for any $\epsilon \in (0,1/2)$ and each $(j,k)\in I_M$,
    there exists a deep ReLU network $\widetilde{\Psi}_{jk}$ satisfying
    \ba
    &L(\widetilde{\Psi}_{jk}) \lesssim \log(2M)\log^2(1/\epsilon),\quad
    W(\widetilde{\Psi}_{jk}) \lesssim \log(2M)\log^3(1/\epsilon),
    \ea
    and
    \ba
    \left\|
    \psi_j\psi_k-\widetilde{\Psi}_{jk}
    \right\|_{\infty}
    \leq \epsilon.
    \ea
    Now, we define an approximator of $K_M^\circ$ as
    \[
    \widetilde{K}_{M,\epsilon}(s,t)
    :=
    \sum_{(j,k)\in I_M}
    c_{jk}\,\widetilde{\Psi}_{jk}(s,t).
    \]
    By parallelization and a final affine output layer, $\widetilde{K}_{M,\epsilon}$ can be represented by a deep ReLU network.
    Here, we note that
    \[
    \sum_{(j,k)\in I_M}|c_{jk}|
    \le
    \left( \sum_{(j,k)\in I_M}\frac{c_{jk}^2}{\rho_j\rho_k}\right)^{1/2}
    \left( \sum_{(j,k)\in I_M}\rho_j\rho_k\right)^{1/2}
    \le Const.\times R,
    \]
    where $Const.$ is a global constant.
    Thus, we obtain
    \begin{align}
    \left\lVert
    K_M^\circ-\widetilde{K}_{M,\epsilon}
    \right\rVert_{2}
    &=
    \left\|
    \sum_{(j,k)\in I_M}
    c_{jk}\left(\psi_j\psi_k-\widetilde{\Psi}_{jk}\right)
    \right\|_{2}
    \leq
    \sum_{(j,k)\in I_M}
    |c_{jk}|
    \left\|
    \psi_j\psi_k-\widetilde{\Psi}_{jk}
    \right\|_{2}
    \nonumber\\
    &\leq
    Const.\times R \max_{(j,k)\in I_M}\left\|\psi_j\psi_k-\widetilde{\Psi}_{jk}\right\|_{\infty} \le Const. \times R\epsilon.
    \label{eq:dnn_hc_bound_reindexed}
    \end{align}
    Combining (\ref{eq:truncation}) and (\ref{eq:dnn_hc_bound_reindexed}), 
    we obtain
    \[
    \|K^\circ - \widetilde{K}_{M,\epsilon}\|_2^2 \le Const. \times R^2\left( M^{-2\alpha} + \epsilon^2\right).
    \]
    Thus, we take $M^{-\alpha}\asymp \epsilon$.
    
    Now, we bound the number of nonzero parameters of $\widetilde K_{M,\epsilon}$.
    By construction, for each $(j,k)$, the network $\widetilde{\Psi}_{jk}$ satisfies
    \ba
    L(\widetilde{\Psi}_{jk}) &\lesssim \log(2M)\log^2(1/\epsilon) \lesssim  \log^3(M),\\
    W(\widetilde{\Psi}_{jk}) &\lesssim \log(2M)\log^3(1/\epsilon) \lesssim  \log^4(M).
    \ea
    Hence, the number of nonzero parameters satisfies
    \[
    S(\widetilde{\Psi}_{jk})
    \lesssim
    L(\widetilde{\Psi}_{jk})\,W(\widetilde{\Psi}_{jk})^2
    \lesssim
    \log^{11}(M).
    \]
    Since the network $\widetilde K_{M,\epsilon}$ is obtained by combining $|I_M|$ such networks in parallel,
    \[
    S(\widetilde K_{M,\epsilon}) \lesssim |I_M|\,\log^{11}(M) \lesssim M\log^{12}(M).
    \]
    Here, we used
    \[
        |I_M| = \#\{(j,k)\in \Nb^2\mid jk\le M \} = \sum_{j=1}^M\lfloor M/j\rfloor \lesssim M\log(M).
    \]
    
    From Theorem~6 in \cite{Bartlett_etal_2019}, the VC dimension of the network class $\Fcal(L,W,S)$ is bounded by
    \[
    \mathrm{VCdim}(\Fcal(L,W,S)) \lesssim LS\log(S).
    \]
    Thus, we consider the network model $\Fcal(L,W,S)$ with $L\asymp \log^3(M)$, $W\asymp M\log^5(M)$, and $S\asymp M\log^{12}(M)$.
    With this choice of $L$, $W$, and $S$, and taking $M \asymp (n \lfloor m/2 \rfloor)^{1/(2\alpha + 1)}$,
    Corollary~\ref{corollary:oracle-without-fixed-point} yields
    \begin{align*}
        \Eb[\| \hat{K}_n - K^\circ \|_2^2] 
        &\le Const.\times  \left\{ M^{-2\alpha} + \frac{1}{n} + \frac{LS\log(S) \log(n \lfloor m/2 \rfloor)}{n \lfloor m/2 \rfloor} \right\} \\
        &\le Const.\times \left\{ \frac{1}{n} + \frac{1}{(n \lfloor m/2 \rfloor)^{2\alpha / (2\alpha + 1)}} \log^{17}(n \lfloor m/2 \rfloor) \right\},
    \end{align*}
    where $Const.$ is independent of $n$ and $m$.
\end{proof}

\section{Technical Lemmas for Theorem~\ref{theorem:oracle}}\label{app:tech-lemmas}

Here, we state several technical lemmas used in the proof of Theorem~\ref{theorem:oracle}.
The following lemmas are obtained via a peeling device \citep{vanDerGeer2000} and a conditioning argument, 
combined with Lemma~\ref{lemma:modified-Talagrand}.
Proofs of these technical lemmas are given in Appendix~\ref{app:proofs-tech-lemmas}.

\begin{lemma}\label{lemma:CI1}
Let $C>1/(40\sqrt{2}B_K)$.
Then, for any $x\ge 0$, with probability at least $1-e^{-x}$, the following inequality holds:
\begin{align*}
&\Biggl|
\|\widehat{K}_n-K^\circ\|_{2}^2
-
\frac{1}{nm(m-1)}
\sum_{i=1}^n\sum_{j\neq k}
\left\{
\widehat{K}_n(T_{ij},T_{ik})-K^\circ(T_{ij},T_{ik})
\right\}^2
\Biggr|\\
&\le
\frac{1}{C}\,
\|\widehat{K}_n-K^\circ\|_{2}^2 + 3200\,C\,B_K^2\,r^\ast + \frac{2B_K^2 x}{n\lfloor m/2 \rfloor}\,(C+22).
\end{align*}
\end{lemma}

\begin{lemma}\label{lemma:CI2}
Let $\tilde{e}_{i}(T_{ij},T_{ik}) := \mathcal{C}_{\beta_U}e_{i}(T_{ij},T_{ik})$.
For any $C_U > 1/(20\sqrt{2}\beta_U)$, the following inequality holds:
\ba
&\Eb\left[\left|\frac{1}{nm(m-1)}\sum_{i=1}^n\sum_{j\neq k}
\tilde{e}_{i}(T_{ij},T_{ik})\left\{\Khat_n(T_{ij},T_{ik}) - K^\circ(T_{ij},T_{ik})\right\} -\la \tilde{e}_i, \Khat_n-\Ktrue \ra\right|\right]\\
&\le
\frac{1}{C_U}\Eb\left[\|\what{K}_n-K^\circ\|_{2}^2\right] + 800C_U^2\beta_U^2 r^\ast  + \frac{\beta_U}{n\lfloor m/2 \rfloor}\left( C_U\beta_U+ 44B_k\right),
\ea
where $\la \tilde{e}_i, K \ra := \int_{\Tcal\times \Tcal} \tilde{e}_i(s,t)K(s,t)\,dP_T^{\otimes 2}(s,t)$.
\end{lemma}

\begin{lemma}\label{lemma:CI3}
Let $e_i^{(V)}(T_{ij},\epsilon_{ik})
:=\mathcal{C}_{\beta_V}(\epsilon_{ik}X_i(T_{ij})) -  \Eb[\mathcal{C}_{\beta_V}(\epsilon_{ik}X_i(T_{ij}))\mid X_i, T_{ij}]$.
For any $C_V > 1/(40\sqrt{2}\beta_V)$, the following inequality holds:
\ba
&\Eb\left[\left|\frac{1}{nm(m-1)}\sum_{i=1}^n\sum_{j\neq k}
e_i^{(V)}(T_{ij},\epsilon_{ik})\left\{K(T_{ij},T_{ik})-  K^\circ(T_{ij},T_{ik})\right\}
\right|\right]\\
&\le
\frac{1}{C_V}\Eb\left[\|\what{K}_n-K^\circ\|_{2}^2\right] + 3200C_V^2\beta_V^2 r^\ast  + \frac{2\beta_V}{n\lfloor m/2 \rfloor}\left( C_V\beta_V+ 26B_k\right).
\ea
\end{lemma}

\begin{lemma}\label{lemma:CI4}
Let $e_i^{(W)}(\epsilon_{ij},\epsilon_{ik})
:=\mathcal{C}_{\beta_W}(\epsilon_{ij}\epsilon_{ik}) - \Eb\left[\mathcal{C}_{\beta_W}(\epsilon_{ij}\epsilon_{ik}) \right]$.
For any $C_V > 1/(40\sqrt{2}\beta_V)$, the following inequality holds:
\ba
&\Eb\left[\left|\frac{1}{nm(m-1)}\sum_{i=1}^n\sum_{j\neq k}
e_i^{(W)}(\epsilon_{ij},\epsilon_{ik})\left\{K(T_{ij},T_{ik})-  K^\circ(T_{ij},T_{ik})\right\}
\right|\right]\\
&\le
\frac{1}{C_W}\Eb\left[\|\what{K}_n-K^\circ\|_{2}^2\right] + 3200C_W^2\beta_W^2 r^\ast  + \frac{2\beta_W}{n\lfloor m/2 \rfloor}\left( C_W\beta_W+ 26B_k\right).
\ea
\end{lemma}
\begin{proof}
The proof of this lemma is omitted, as it is essentially the same as that of Lemma~\ref{lemma:CI3}.
\end{proof}

\section{Proofs of technical lemmas}\label{app:proofs-tech-lemmas}

\subsection{Proof of Lemma~\ref{lemma:CI1}}

\begin{proof}[Proof of Lemma~\ref{lemma:CI1}]
The proof follows the localization analysis (\citealp{BartlettEtAl2005,BlanchardEtAl2008}) 
combining with Lemma~\ref{lemma:modified-Talagrand}.

First, we apply the peeling device (e.g., \cite{vanDerGeer2000}) to derive a tight inequality.
We consider the usual symmetrization technique (e.g., \citealp{vanDerVaartWellner1996}).
Let $\{T_{ij}'\}$ be an independent copy of $\{T_{ij}\}$.
For simplicity of notation, let $j_1=j$ and $j_2:= \lfloor m/2 \rfloor+j$.
We have
\ba
&\int_{\Tcal \times \Tcal} (K-K^\circ)^2\,d(P_T\otimes P_T) 
- \left\{K(T_{ij_1},T_{ij_2}) -  K^\circ(T_{ij_1},T_{ij_2})\right\}^2\\
&=
\mathbb{E}_{T'}\left[ 
\{ K(T_{ij_1}',T_{ij_2}')-K^\circ(T_{ij_1}',T_{ij_2}')\}^2 
- \{K(T_{ij_1},T_{ij_2}) -  K^\circ(T_{ij_1},T_{ij_2})\}^2
\right].
\ea
Hence, 
we have
{\small
\ba
&\frac{\int_{\Tcal \times \Tcal} (K-K^\circ)^2\,d(P_T\otimes P_T) 
- \left\{K(T_{ij_1},T_{ij_2})-  K^\circ(T_{ij_1},T_{ij_2})\right\}^2}
{\int_{\Tcal \times \Tcal} (K-K^\circ)^2\,d(P_T\otimes P_T) + r}\\
&=
\mathbb{E}_{T'}\left[
\frac{ \left\{K(T_{ij_1}',T_{ij_2}') -  K^\circ(T_{ij_1}',T_{ij_2}')\right\}^2- \left\{K(T_{ij_1},T_{ij_2}) -  K^\circ(T_{ij_1},T_{ij_2})\right\}^2}
{\int_{\Tcal \times \Tcal} (K-K^\circ)^2\,d(P_T\otimes P_T) + r}
\right].
\ea
}
Therefore, we have
{\footnotesize
\ba
&\sup_{K\in \Kcal}\left|
\sum_{i = 1}^n  \sum_{j=1}^{\lfloor m/2 \rfloor} 
\frac{\int (K-K^\circ)^2\,d(P_T\otimes P_T) - \left\{K(T_{ij_1},T_{ij_2}) -  K^\circ(T_{ij_1},T_{ij_2})\right\}^2}
{\int_{\Tcal \times \Tcal} (K-K^\circ)^2\,d(P_T\otimes P_T) + r}
 \right|\\
 &\le 
 \sup_{K\in \Kcal}\left|
\mathbb{E}_{T'}\left[
\sum_{i = 1}^n  \sum_{j=1}^{\lfloor m/2 \rfloor} 
\frac{ \left\{K(T_{ij_1}',T_{ij_2}') -  K^\circ(T_{ij_1}',T_{ij_2}')\right\}^2- \left\{K(T_{ij_1},T_{ij_2}) -  K^\circ(T_{ij_1},T_{ij_2})\right\}^2}
{\int_{\Tcal \times \Tcal} (K-K^\circ)^2\,d(P_T\otimes P_T) + r}
\right]
\right|\\
&\le
\mathbb{E}_{T'}\left[ \sup_{K\in \Kcal}\left|
\sum_{i = 1}^n  \sum_{j=1}^{\lfloor m/2 \rfloor} 
\frac{ \left\{K(T_{ij_1}',T_{ij_2}') -  K^\circ(T_{ij_1}',T_{ij_2}')\right\}^2- \left\{K(T_{ij_1},T_{ij_2}) -  K^\circ(T_{ij_1},T_{ij_2})\right\}^2}
{\int_{\Tcal \times \Tcal} (K-K^\circ)^2\,d(P_T\otimes P_T) + r}
\right|\right],
\ea
}
and we obtain 
{\footnotesize
\ba
&\Eb_T\left[ 
\sup_{K\in \Kcal}\left|
\sum_{i = 1}^n  \sum_{j=1}^{\lfloor m/2 \rfloor} 
\frac{\int_{\Tcal \times \Tcal} (K-K^\circ)^2\,d(P_T\otimes P_T) - \left\{K(T_{ij_1},T_{ij_2})-  K^\circ(T_{ij_1},T_{ij_2})\right\}^2}
{\int_{\Tcal \times \Tcal} (K-K^\circ)^2\,d(P_T\otimes P_T) + r}
 \right|
\right]\\
&\le
\mathbb{E}_{T,T'}\left[ \sup_{K\in \Kcal}\left|
\sum_{i = 1}^n  \sum_{j=1}^{\lfloor m/2 \rfloor} 
\frac{ \left\{K(T_{ij_1}',T_{ij_2}') -  K^\circ(T_{ij_1}',T_{ij_2}')\right\}^2- \left\{K(T_{ij_1},T_{ij_2}) -  K^\circ(T_{ij_1},T_{ij_2})\right\}^2}
{\int_{\Tcal \times \Tcal} (K-K^\circ)^2\,d(P_T\otimes P_T) + r}
\right|\right]\\
&\le
2\mathbb{E}_{T,T',\sigma}\left[ \sup_{K\in \Kcal}\left|
\sum_{i = 1}^n  \sum_{j=1}^{\lfloor m/2 \rfloor} 
\sigma_{ij} \frac{ \{K(T_{ij_1},T_{ij_2}) -  K^\circ(T_{ij_1},T_{ij_2})\}^2}
{\|K-K^\circ\|_{2}^2 + r}
 \right|\right],
\ea
}
where $\{\sigma_{ij}\}$ is a Rademacher sequence, which is independent from $\{T_{ij}\}$.

Let $r$ be a real number not less than $r^\ast$ (i.e., $r\ge r^\ast$).
Then, for $s>1$, since
\[
\Kcal = \Kcal(r) \cup \left( \bigcup_{k=1}^\infty \{ \Kcal(s^k r)\setminus \Kcal(s^{k-1}r)\}\right),
\]
we have
{\small 
\ba
&\mathbb{E}_{T,T',\sigma}\left[ \sup_{K\in \Kcal}\left|
\sum_{i = 1}^n  \sum_{j=1}^{\lfloor m/2 \rfloor} 
\sigma_{ij} \frac{ \{K(T_{ij_1},T_{ij_2}) -  K^\circ(T_{ij_1},T_{ij_2})\}^2}
{\|K-K^\circ\|_{2}^2 + r}
 \right|\right]\\
&\le
\mathbb{E}_{T,T',\sigma}\left[ \sup_{K\in \Kcal(r)}\left|
\sum_{i = 1}^n  \sum_{j=1}^{\lfloor m/2 \rfloor} 
\sigma_{ij} \frac{ \{K(T_{ij_1},T_{ij_2}) -  K^\circ(T_{ij_1},T_{ij_2})\}^2}
{\|K-K^\circ\|_{2}^2 + r}
 \right|\right]\\
&\quad +
\sum_{l=1}^\infty \mathbb{E}_{T,T',\sigma}\left[ \sup_{K\in  \Kcal(s^l r)\setminus \Kcal(s^{l-1}r)}\left|
\sum_{i = 1}^n  \sum_{j=1}^{\lfloor m/2 \rfloor} 
\sigma_{ij} \frac{ \{K(T_{ij_1},T_{ij_2}) -  K^\circ(T_{ij_1},T_{ij_2})\}^2}
{\|K-K^\circ\|_{2}^2 + r}
\right|\right]\\
&\le
\frac{1}{r}\mathbb{E}_{T,T',\sigma}\left[ \sup_{K\in \Kcal(r)}\left|
\sum_{i = 1}^n  \sum_{j=1}^{\lfloor m/2 \rfloor} 
\sigma_{ij} \{K(T_{ij_1},T_{ij_2}) -  K^\circ(T_{ij_1},T_{ij_2})\}^2
 \right|\right]\\
&\quad +
\sum_{l=1}^\infty \frac{1}{(s^{l-1}+1)r}\,
\mathbb{E}\left[ \sup_{K\in  \Kcal(s^l r)\setminus \Kcal(s^{l-1}r)}\left|
\sum_{i = 1}^n  \sum_{j=1}^{\lfloor m/2 \rfloor} 
\sigma_{ij} \{K(T_{ij_1},T_{ij_2}) -  K^\circ(T_{ij_1},T_{ij_2})\}^2
 \right|\right] 
\ea
}
Moreover, by the comparison theorem (Theorem 4.2 in \cite{}),
for all $r\ge r^\ast$,
\ba
&\mathbb{E}\left[ \sup_{K\in \Kcal(r)}\left|
\frac{1}{n\lfloor m/2 \rfloor}\sum_{i = 1}^n  \sum_{j=1}^{\lfloor m/2 \rfloor} 
\sigma_{ij} \left\{K(T_{ij_1},T_{ij_2}) -  K^\circ(T_{ij_1},T_{ij_2})\right\}^2
\right|\right]\\
&\le
4B_K  
\Eb\left[ \sup_{K\in \Kcal(r)}\left|
\frac{1}{n\lfloor m/2 \rfloor}\sum_{i = 1}^n  \sum_{j=1}^{\lfloor m/2 \rfloor} 
\sigma_{ij} \left\{K(T_{ij_1},T_{ij_2}) -  K^\circ(T_{ij_1},T_{ij_2})\right\}
\right|\right]\\
&\le
4B_K\phi(r)
\ea
holds.

Combining these bounds, we obtain
\ba
&\frac{4B_K \phi(r)}{r} + \sum_{l=1}^\infty \frac{4B_K\phi(s^l r)}{(s^{l-1}+1)r}
=\frac{4B_K}{r}\left\{ \phi(r)+ \sum_{l=1}^\infty \frac{\phi(s^l r)}{s^{l-1}+1} \right\}\\
&\le 
\frac{4B_K}{r}\left\{ \phi(r)+ \sum_{l=1}^\infty \frac{s^{l/2}\phi(r)}{s^{l-1}+1} \right\}
=
\frac{4B_K\phi(r)}{r}\left\{ 1+ \sum_{l=0}^\infty \frac{s^{(l+1)/2}}{s^{l}+1} \right\}\\
&\le
\frac{4B_K\phi(r)}{r}\left\{ 1+ s^{1/2}\left( \frac{1}{2} + \sum_{l=1}^\infty s^{-l/2}\right) \right\}
=
\frac{4B_K\phi(r)}{r}\left\{ 1+ s^{1/2}\left( \frac{1}{2} + \frac{1}{s^{1/2}-1}\right) \right\}
\ea
Here, the right-hand side is minimized at $s = 3 + 2\sqrt{2}$, and thus the upper bound satisfies
\[
\frac{4B_K\phi(r)}{r}\frac{5+2\sqrt{2}}{2}
\le
\frac{4B_K\phi(r)}{r}\frac{8}{2} = \frac{16B_K\phi(r)}{r}.
\]

Therefore,
\ba
&\Eb_T\left[ 
\sup_{K\in \Kcal}\left|
\frac{1}{n\lfloor m/2 \rfloor}
\sum_{i = 1}^n  \sum_{j=1}^{\lfloor m/2 \rfloor} 
\frac{\|K-K^\circ\|_{2}^2  - \left\{K(T_{ij_1},T_{ij_2})-  K^\circ(T_{ij_1},T_{ij_2})\right\}^2}
{\|K-K^\circ\|_{2}^2  + r}
 \right|
\right]\\
&\le 
\frac{32B_K \phi(r)}{r}
\ea
holds.

Next, we verify the conditions needed to apply the modified Talagrand inequality (Lemma~\ref{lemma:modified-Talagrand}).
Define
\[
f_{K}(\cdot,\star)
:=
\frac{\|K-K^\circ\|_{2}^2 - \left\{K(\cdot,\star) -  K^\circ(\cdot,\star)\right\}^2}
{\|K-K^\circ\|_{2}^2 + r},
\]
and set $\mathcal{F}:=\{f_K\mid K\in \Kcal\}$.
First, it is clear that $\Eb\left[f_{K}(T_{i1},T_{i2}) \right] = 0$.
By the arithmetic--geometric mean inequality, for $j\neq k$,
\ba
\Eb\left[f_{K}^2(T_{ij},T_{ik})\right]
&=\Eb\left[
\left| \frac{\|K-K^\circ\|_{2}^2 - \left\{K(T_{ij},T_{ik}) -  K^\circ(T_{ij},T_{ik})\right\}^2}
{\|K-K^\circ\|_{2}^2 + r} \right|^2
\right]\\
&\le \frac{\Eb\left[\left\{K(T_{ij},T_{ik}) -  K^\circ(T_{ij},T_{ik})\right\}^4\right]}
{4r\|K-K^\circ\|_{2}^2}\\
&\le \frac{\Eb\left[(2B_K)^2\left\{K(T_{ij},T_{ik}) -  K^\circ(T_{ij},T_{ik})\right\}^2\right]}
{4r\|K-K^\circ\|_{2}^2}
\le 
\frac{B_K^2}{r},
\ea
and hence we can take $\sigma^2:= B_K^2/r$.
Also, we have
\ba
\left|f_{K}(T_{ij},T_{ik}) \right|
&\le
\left|\frac{\|K-K^\circ\|_{2}^2 - \left\{K(T_{ij},T_{ik}) -  K^\circ(T_{ij},T_{ik})\right\}^2}
{\|K-K^\circ\|_{2}^2 + r}\right|
\le
\frac{4B_K^2}{r}=:U.
\ea
Therefore, for all $x \ge 0$, with probability at least $1-e^{-x}$, the following holds:
{\small
\ba
&\sup_{K\in \Kcal} \left| \frac{1}{nm(m-1)}\sum_{i=1}^n \sum_{j\neq k}f_{K}(T_{ij},T_{ik})  \right| \\
&\le (1+\alpha)\Eb\left[ \sup_{K\in \Kcal}\left| \frac{1}{n\lfloor m/2 \rfloor}\sum_{i = 1}^n  \sum_{j=1}^{\lfloor m/2 \rfloor} f_K\big(T_{\pi(j)},T_{\pi(\lfloor m/2 \rfloor+j)}\big)\right|\right]\\
&\quad+ \sqrt{\frac{2B_K^2}{n\lfloor m/2 \rfloor}\frac{x}{r}} + \left( \frac{1}{3}+\frac{1}{\alpha} \right)\frac{4B_K^2}{r}\frac{x}{n\lfloor m/2 \rfloor}\\
&\le
(1+\alpha)\frac{32B_K \phi(r)}{r}+ \sqrt{\frac{2B_K^2}{n\lfloor m/2 \rfloor}\frac{x}{r}} 
+ \left( \frac{3}{2}+\frac{1}{\alpha} \right)\frac{4B_K^2}{r}\frac{x}{n\lfloor m/2 \rfloor}.
\ea
}
That is, for all $x \ge 0$, with probability at least $1-e^{-x}$,
the following holds:
\ba
\forall K \in \Kcal ;\;
&\left|\frac{\|K-K^\circ\|_{2}^2 - \frac{1}{nm(m-1)}\sum_{i=1}^n \sum_{j\neq k}\left\{K(T_{ij},T_{ik}) - K^\circ(T_{ij},T_{ik})\right\}^2}
{\|K-K^\circ\|_{2}^2 + r}\right|\\
&\le
\inf_{\alpha >0}\left\{
32(1+\alpha)B_K\sqrt{\frac{r^\ast}{r}}+ \sqrt{\frac{2B_K^2}{n\lfloor m/2 \rfloor}\frac{x}{r}} 
+ \left( \frac{3}{2}+\frac{1}{\alpha} \right)\frac{4B_K^2}{r}\frac{x}{n\lfloor m/2 \rfloor}
 \right\}.
\ea

Let
\[
A_1 := 32(1+\alpha)B_K\sqrt{r^\ast}+ \sqrt{\frac{2B_K^2x}{n\lfloor m/2 \rfloor}},\quad
A_2:= \left( \frac{3}{2}+\frac{1}{\alpha} \right)\frac{4B_K^2x}{n\lfloor m/2 \rfloor}.
\]
Then it suffices to find $r$ satisfying
\ba
A_1 \frac{1}{\sqrt{r}} + A_2\frac{1}{r} \le \frac{1}{C}.
\ea
It is enough to find an $r$ such that $r \ge C^2 A_1^2  + 2CA_2$.
Since we have
\ba
C^2 A_1^2  + 2CA_2
\le 
2(32)^2(1+\alpha)^2C^2B_K^2r^\ast + 4B_K^2\frac{x}{n\lfloor m/2 \rfloor}\left\{C^2 + 2C \left( \frac{3}{2}+\frac{1}{\alpha} \right) \right\},
\ea
and then taking $\alpha = 1/4$ yields that the above condition is satisfied for any $r$ such that
\ba
2B_K^2\left\{ (32)^2(5/4)^2C^2 r^\ast  + \frac{x}{n\lfloor m/2 \rfloor}(C^2 + 22C) \right\}\le r.
\ea
Note that if $C > 1/(40\sqrt{2}B_K)$, then $r > r^\ast$ holds.
Therefore, for any $x>0$ and any $C > 1/(40\sqrt{2}B_K)$, with probability at least $1-e^{-x}$,
the following holds:
\ba
\forall K \in \Kcal ;\;
&\left| \|K-K^\circ\|_{2}^2 - \frac{1}{nm(m-1)}\sum_{i=1}^n \sum_{j\neq k}\left\{K(T_{ij},T_{ik}) - K^\circ(T_{ij},T_{ik})\right\}^2\right|\\
&\le
\frac{1}{C}\|K-K^\circ\|_{2}^2 + 3200CB_K^2 r^\ast  + \frac{2B_K^2x}{n\lfloor m/2 \rfloor}(C + 22).
\ea
\end{proof}

\subsection{Proof of Lemma~\ref{lemma:CI2}}

\begin{proof}[Proof of Lemma~\ref{lemma:CI2}]
First, we work conditionally on the stochastic processes $X_1,\dots, X_n$.
Under this conditioning, $e_i$ can be regarded as a fixed function.
Let
\[
f_i(\cdot,\star) 
:= \frac{\tilde{e}_{i}(\cdot,\star)\left\{K(\cdot,\star) - K^\circ(\cdot,\star)\right\} -\la \tilde{e}_i, K-\Ktrue \ra}{\|K-\Ktrue\|_2^2 + r}.
\]
We verify the conditions needed to apply the modified Talagrand inequality (Lemma~\ref{lemma:modified-Talagrand}).
By definition, for $j\neq k$, we have $\Eb_{T}[ f_i(T_{ij},T_{ik})] = 0$.
Moreover,
\ba
|f_i(T_{ij},T_{ik})|
&=
\left|
 \frac{\tilde{e}_{i}(T_{ij},T_{ik})\left\{K(T_{ij},T_{ik}) - K^\circ(T_{ij},T_{ik})\right\} -\la \tilde{e}_i, K-\Ktrue \ra}{\|K-\Ktrue\|_2^2 + r}
\right|
\le
\frac{4\beta_U B_K}{r},
\\
\Eb_{T}[f_i^2(T_{ij},T_{ik})] 
&=
\Eb_T\left[
\left|
 \frac{\tilde{e}_{i}(T_{ij},T_{ik})\left\{K(T_{ij},T_{ik}) - K^\circ(T_{ij},T_{ik})\right\} -\la \tilde{e}_i, K-\Ktrue \ra}{\|K-\Ktrue\|_2^2 + r}
\right|^2
\right]
\\
&\le
\frac{\beta_U^2}{4r\|K-\Ktrue\|_2^2}
\Eb_T\left[
\left|K(T_{ij},T_{ik}) - K^\circ(T_{ij},T_{ik})\right|^2
\right]
=  \frac{\beta_U^2}{4r}.
\ea
By Lemma~\ref{lemma:modified-Talagrand}, for any $x\ge 0$ and any $\alpha>0$, with probability at least $1-e^{-x}$,
\ba
&\sup_{f\in \Fcal} \left| \frac{1}{nm(m-1)}\sum_{i=1}^n \sum_{j\neq k}f_{i}(T_{ij},T_{ik})  \right| \\
&\le (1+\alpha)\Eb_{T\mid X}\left[ \sup_{f\in \Fcal}\left| \frac{1}{n\lfloor m/2 \rfloor}\sum_{i = 1}^n  \sum_{j=1}^{\lfloor m/2 \rfloor} 
f_i\big(T_{ij},T_{i(\lfloor m/2 \rfloor+j)}\big)\right|\right] \\
&\quad+
\sqrt{\frac{1}{n\lfloor m/2 \rfloor}\frac{\beta_U^2}{2r}x}
+\left( \frac{3}{2} + \frac{1}{\alpha}\right)\frac{4\beta_UB_K}{rn\lfloor m/2 \rfloor}x
\ea
holds.

By the symmetrization argument, we have
\ba
&\Eb_{T\mid X}\left[ \sup_{K\in \Kcal} \left|\sum_{i=1}^n\sum_{j=1}^{\lfloor m/2 \rfloor}\frac{\tilde{e}_{i}(T_{ij_1},T_{ij_2})
\left\{K(T_{ij_1},T_{ij_2})-  K^\circ(T_{ij_1},T_{ij_2})\right\}- \la \tilde{e}_i, K-\Ktrue \ra}{\|K-\Ktrue\|_2^2 + r}\right|\right] \\
&\le
2\Eb_{T\mid X}\left[ \sup_{K\in \Kcal} \left|\sum_{i=1}^n\sum_{j=1}^{\lfloor m/2 \rfloor}
\frac{\sigma_{ij}\tilde{e}_{i}(T_{ij_1},T_{ij_2})\left\{K(T_{ij_1},T_{ij_2})-  K^\circ(T_{ij_1},T_{ij_2})\right\}}
{\|K-\Ktrue\|_2^2 + r}\right|\right].
\ea
By the comparison theorem, it follows that
\ba
&\Eb_{T\mid X}\left[ \sup_{K\in \Kcal(r)} \left|\frac{1}{n\lfloor m/2 \rfloor}\sum_{i=1}^n\sum_{j=1}^{\lfloor m/2 \rfloor}
\sigma_{ij}\tilde{e}_{i}(T_{ij_1},T_{ij_2})\left\{K(T_{ij_1},T_{ij_2})-  K^\circ(T_{ij_1},T_{ij_2})\right\}\right|\right] \\
&\le
2\beta_U
\Eb\left[ \sup_{K\in \Kcal(r)} \left|\frac{1}{n\lfloor m/2 \rfloor}\sum_{i=1}^n\sum_{j=1}^{\lfloor m/2 \rfloor}
\sigma_{ij}\left\{K(T_{ij_1},T_{ij_2})-  K^\circ(T_{ij_1},T_{ij_2})\right\}\right|\right]
=2\beta_U\phi(r).
\ea
As in Lemma~\ref{lemma:CI1}, we obtain
\ba
&\Eb_{T\mid X}\left[ \sup_{K\in \Kcal} \left|\frac{1}{n\lfloor m/2 \rfloor}\sum_{i=1}^n\sum_{j=1}^{\lfloor m/2 \rfloor}\frac{\tilde{e}_{i}(T_{ij_1},T_{ij_2})
\left\{K(T_{ij_1},T_{ij_2})-  K^\circ(T_{ij_1},T_{ij_2})\right\}- \la \tilde{e}_i, K-\Ktrue \ra}{\|K-\Ktrue\|_2^2 + r}\right|\right] \\
&\le 
\frac{16\beta_U\phi(r)}{r}.
\ea

Therefore, the upper bound is given by
\[
16(1+\alpha)\beta_U\sqrt{\frac{r^\ast}{r}}
+\sqrt{\frac{1}{n\lfloor m/2 \rfloor}\frac{\beta_U^2}{2r}x}
+\left( \frac{3}{2} + \frac{1}{\alpha}\right)\frac{4\beta_UB_K}{rn\lfloor m/2 \rfloor}x.
\]
Consequently, as in Lemma~\ref{lemma:CI1},
\ba
C^2A_1^2 +2CA_2 
&\le 
2(16)^2(1+\alpha)^2C^2\beta_U^2 r^\ast + \frac{\beta_U^2}{n\lfloor m/2 \rfloor}C^2x
+2C\left( \frac{3}{2} + \frac{1}{\alpha}\right)\frac{4\beta_UB_K}{n\lfloor m/2 \rfloor}x \\
&=
2(16)^2(1+\alpha)^2C^2\beta_U^2 r^\ast 
+ \frac{\beta_Ux}{n\lfloor m/2 \rfloor }\left\{ C^2\beta_U+ 8CB_k\left( \frac{3}{2} + \frac{1}{\alpha}\right) \right\}
\ea
and hence, taking $\alpha = 1/4$, the upper bound becomes $1/C$ for any $r$ satisfying
\[
2(20)^2C^2\beta_U^2 r^\ast  + \frac{\beta_Ux}{n\lfloor m/2 \rfloor }\left( C^2\beta_U+ 44CB_k\right)\le r.
\]
If $C > 1/(20\sqrt{2}\beta_U)$, then $r>r^\ast$ holds.
Therefore, for any $x>0$ and any $C > 1/(20\sqrt{2}\beta_U)$, with probability at least $1-e^{-x}$,
\ba
\forall K \in \Kcal;\;
&\left|\frac{1}{n\lfloor m/2 \rfloor}\sum_{i=1}^n\sum_{j=1}^{\lfloor m/2 \rfloor}\tilde{e}_{i}(T_{ij},T_{ik})\left\{K(T_{ij},T_{ik}) - K^\circ(T_{ij},T_{ik})\right\} -\la \tilde{e}_i, K-\Ktrue \ra\right| \\
&\le
\frac{1}{C}\|K-K^\circ\|_{2}^2 + 800C^2\beta_U^2 r^\ast  + \frac{\beta_Ux}{n\lfloor m/2 \rfloor }\left( C\beta_U+ 44B_k\right)
\ea
holds.
\end{proof}

\subsection{Proof of Lemma~\ref{lemma:CI3}}

\begin{proof}[Proof of Lemma~\ref{lemma:CI3}]
We first condition on the stochastic process $X_1,\dots,X_n$. 
Also, let $Z_{ij} := (T_{ij},\epsilon_{ij})$, and define
\ba
f_i(Z_{ij},Z_{ik})
:=
\frac{e_i^{(V)}(T_{ij},\epsilon_{ik})\left\{K(T_{ij},T_{ik}) - K^\circ(T_{ij},T_{ik})\right\}}
{\|K-K^\circ\|_{2}^2+r}.
\ea
For $j\neq k$, we have $\mathbb{E}_b[f_i(Z_{ij},Z_{ik})] = 0$ and
\ba
\left| f_i(Z_{ij},Z_{ik})\right|
&=
\left|
\frac{e_i^{(V)}(T_{ij},\epsilon_{ik})\left\{K(T_{ij},T_{ik}) - K^\circ(T_{ij},T_{ik})\right\}}
{\|K-K^\circ\|_{2}^2+r}
\right|
\le
\frac{4B_K\beta_V}{r},\\
\mathbb{E}_b\left[\left| f_i(Z_{ij},Z_{ik})\right|^2 \right]
&=
\mathbb{E}\left[
\left| \frac{e_i^{(V)}(T_{ij},\epsilon_{ik})\left\{K(T_{ij},T_{ik}) - K^\circ(T_{ij},T_{ik})\right\}}{\|K-K^\circ\|_{2}^2+r}\right|^2
\right]\\
&\le
\frac{4\beta_V^2}{4r\|K-K^\circ\|_{2}^2}\mathbb{E}\left[\left|K(T_{ij},T_{ik}) - K^\circ(T_{ij},T_{ik})\right|^2\right]
= \frac{\beta_V^2}{r}.
\ea
Further, noting that $\mathbb{E}_b\bigl[e_i^{(V)}(T_{ij},\epsilon_{ik})\mid X_i, T_{ij}\bigr] = 0$,
the symmetrization argument yields
\ba
&\mathbb{E}_{Z\mid X}\left[ \sup_{K\in \Kcal} \left|\sum_{i=1}^n\sum_{j=1}^{\lfloor m/2 \rfloor}
\frac{e_i^{(V)}(T_{ij_1},\epsilon_{ij_2})\left\{K(T_{ij_1},T_{ij_2})-  K^\circ(T_{ij_1},T_{ij_2})\right\}}{\|K-\Ktrue\|_2^2 + r}\right|\right]\\
&\le
2\mathbb{E}_{Z\mid X}\left[ \sup_{K\in \Kcal} \left|\sum_{i=1}^n\sum_{j=1}^{\lfloor m/2 \rfloor}
\frac{\sigma_{ij}e_i^{(V)}(T_{ij_1},\epsilon_{ij_2})\left\{K(T_{ij_1},T_{ij_2})-  K^\circ(T_{ij_1},T_{ij_2})\right\}}
{\|K-\Ktrue\|_2^2 + r}\right|\right].
\ea
By the comparison theorem,
\ba
&\mathbb{E}\left[ \sup_{K\in \Kcal(r)} \left|\frac{1}{n\lfloor m/2 \rfloor}\sum_{i=1}^n\sum_{j=1}^{\lfloor m/2 \rfloor}
\sigma_{ij}e_i^{(V)}(T_{ij_1},\epsilon_{ij_2})\left\{K(T_{ij_1},T_{ij_2})-  K^\circ(T_{ij_1},T_{ij_2})\right\}\right|\right]\\
&\le
4\beta_V
\mathbb{E}\left[ \sup_{K\in \Kcal(r)} \left|\frac{1}{n\lfloor m/2 \rfloor}\sum_{i=1}^n\sum_{j=1}^{\lfloor m/2 \rfloor}
\sigma_{ij}\left\{K(T_{ij_1},T_{ij_2})-  K^\circ(T_{ij_1},T_{ij_2})\right\}\right|\right]
=4\beta_V\phi(r).
\ea
Following the proof of Lemma~\ref{lemma:CI1}, we obtain
\ba
&\mathbb{E}\left[ \sup_{K\in \Kcal} \left|\frac{1}{n\lfloor m/2 \rfloor}\sum_{i=1}^n\sum_{j=1}^{\lfloor m/2 \rfloor}
\frac{e_i^{(V)}(T_{ij_1},\epsilon_{ij_2})
\left\{K(T_{ij_1},T_{ij_2})-  K^\circ(T_{ij_1},T_{ij_2})\right\}}{\|K-\Ktrue\|_2^2 + r}\right|\right]
\le 
\frac{32\beta_V\phi(r)}{r}.
\ea

By Lemma~\ref{lemma:modified-Talagrand}, for any $x\ge 0$ and any $\alpha >0$, with probability at least $1-e^{-x}$,
\ba
\sup_{f\in \Fcal} \left| \frac{1}{n\lfloor m/2 \rfloor}\sum_{i=1}^n\sum_{j=1}^{\lfloor m/2 \rfloor}f_{i}(Z_{ij_1},Z_{i\pi_{2j}})  \right|
&\le (1+\alpha)\mathbb{E}_{Z\mid X}\left[ \sup_{f\in \Fcal}\left| \frac{1}{n\lfloor m/2 \rfloor}\sum_{i = 1}^n  \sum_{j=1}^{\lfloor m/2 \rfloor} 
f_i\bigl(T_{ij_1},T_{ij_2}\bigr)\right|\right]\\
&\quad+
\sqrt{\frac{1}{n\lfloor m/2 \rfloor}\frac{2\beta_V^2}{r}x}
+\left( \frac{3}{2} + \frac{1}{\alpha}\right)\frac{4\beta_VB_K}{rn\lfloor m/2 \rfloor}x
\ea
holds, and the upper bound is given by
\[
32(1+\alpha)\beta_U\sqrt{\frac{r^\ast}{r}}
+\sqrt{\frac{1}{n\lfloor m/2 \rfloor}\frac{2\beta_V^2}{r}x}
+\left( \frac{3}{2} + \frac{1}{\alpha}\right)\frac{4\beta_VB_K}{rn\lfloor m/2 \rfloor}x.
\]
Therefore, as in Lemma~\ref{lemma:CI1},
\ba
C^2A_1^2 +2CA_2 
&\le 
2(32)^2(1+\alpha)^2C^2\beta_U^2 r^\ast + \frac{2\beta_V^2}{n\lfloor m/2 \rfloor}C^2x
+2C\left( \frac{3}{2} + \frac{1}{\alpha}\right)\frac{4\beta_VB_K}{n\lfloor m/2 \rfloor}x\\
&=
2(32)^2(1+\alpha)^2C^2\beta_V^2 r^\ast 
+ \frac{2\beta_Vx}{n\lfloor m/2 \rfloor }\left\{ C^2\beta_V+ 4CB_k\left( \frac{3}{2} + \frac{1}{\alpha}\right) \right\}.
\ea
Hence, taking $\alpha = 1/4$, the above upper bound is at most $1/C$ for any $r$ satisfying
\[
2(40)^2C^2\beta_U^2 r^\ast  + \frac{\beta_Ux}{n\lfloor m/2 \rfloor }\left( C^2\beta_U+ 26CB_k\right)\le r.
\]
Here, if $C > 1/(40\sqrt{2}\beta_V)$, then $r>r^\ast$ holds. 
Therefore, for any $x>0$ and any $C > 1/(40\sqrt{2}\beta_V)$, with probability at least $1-e^{-x}$,
we obtain
\ba
\forall K \in \Kcal,\;\;
&\left|\frac{1}{n\lfloor m/2 \rfloor}\sum_{i=1}^n\sum_{j=1}^{\lfloor m/2 \rfloor}
e_i^{(V)}(T_{ij_1},\epsilon_{ij_2})\left\{K(T_{ij_1},T_{ij_2})-  K^\circ(T_{ij_1},T_{ij_2})\right\}\right|\notag\\
&\le
\frac{1}{C}\|K-K^\circ\|_{2}^2 + 3200C^2\beta_V^2 r^\ast  + \frac{2\beta_Vx}{n\lfloor m/2 \rfloor }\left( C\beta_V+ 26B_k\right).
\ea
This completes the proof.
\end{proof}

\end{appendix}

\begin{funding}
The first author was supported by JSPS KAKENHI Grants (JP24K14855, JP25K15032, JP25K21806).
The second author was supported in part by JST BOOST (JPMJBS2402).
\end{funding}



\bibliographystyle{imsart-nameyear} 
\bibliography{fda_paper}       


\end{document}